\RequirePackage{iftex}
\ifXeTeX
  \RequirePackage{scrlfile}
  \PreventPackageFromLoading{breakurl}
\fi
\documentclass[final,hidelinks,nohypdvips]{siamart251216}
\ifXeTeX

\fi
\usepackage{amssymb, amsfonts, mathtools}
\usepackage{mathrsfs}
\usepackage{url}
\usepackage{microtype}
\usepackage{needspace}

\usepackage{enumitem}
\usepackage{scrlfile}
\usepackage{csquotes}
\usepackage[numbers,sort&compress]{natbib}

\usepackage{bookmark}
\usepackage{booktabs,tabularx}
\newcolumntype{Y}{>{\raggedright\arraybackslash}X}

\theoremstyle{plain}
\theoremheaderfont{\normalfont\scshape}
\theorembodyfont{\normalfont\itshape}
\theoremsymbol{}
\newtheorem{thm}{Theorem}[section]

\newtheorem{lem}{Lemma}[section]
\newtheorem{prop}{Proposition}[section]
\newtheorem{exam}{Example}[section]
\theorembodyfont{\normalfont}
\newtheorem{defn}{Definition}[section]
\newtheorem{ass}{Assumption}[section]
\newtheorem{rmk}{Remark}[section]

\newcommand{\UC}{\textnormal{(UC)}}

\newcommand{\llbracket}{\mathopen{[\![}}
\newcommand{\rrbracket}{\mathclose{]\!]}}
\renewcommand{\theequation}{\thesection.\arabic{equation}}
\makeatletter\@addtoreset{equation}{section}\makeatother

\begin{document}
\raggedbottom

\title{Indefinite Stochastic LQ Optimal Control for Jump-Diffusion Systems with Random Coefficients
	\thanks{This work was supported by the National Natural Science Foundation of China (No.12271158), the Natural Science Foundation of Zhejiang Province for Distinguished Young Scholar (No.LZ22A010005), the Natural Science Foundation of Huzhou City (No.2023YZ46) and the Postgraduate Research and Innovation Project of Huzhou Normal University (No.2024KYCX66).}}

\author{Xinyu Ma\thanks{School of Mathematical Sciences, Zhejiang Normal University, Jinhua 321004, P.R. China (\email{mxyzfd@163.com}).}
\and Qingxin Meng\thanks{Corresponding author. Department of Mathematical Sciences, Huzhou Normal University, Zhejiang 313000, P.R. China (\email{mqx@huznu.edu.cn}).}}
\headers{Indefinite stochastic control with jumps}{X. Ma and Q. Meng}

\maketitle

\begin{abstract}
This paper studies finite-horizon stochastic linear-quadratic control for systems driven by Brownian motion and a finite-activity Poisson random measure, with bounded random coefficients and indefinite cost weights. The control is allowed to enter the drift, diffusion and jump coefficients. Under uniform convexity of the cost functional, we construct a bounded strongly regular solution of the stochastic Riccati equation with jumps. This solution yields a feedback representation of the unique open-loop optimal control. The proof uses the quadratic value field, continuous fundamental matrices on intervals between Poisson times, and a comparison of post-jump value kernels under the compensator measure. These arguments establish the uniform positivity of the effective control-weight matrix and identify the entire finite-variation part of the value kernel. The jump multipliers of the zero-control and optimal closed-loop systems may be singular. The value solution is maximal among bounded strongly regular solutions and unique in the feedback-admissible class. Explicit examples and an indefinite portfolio problem illustrate the results.

\end{abstract}

\section{Introduction}

In this paper, we study the relation between uniform convexity, Riccati solvability and optimal feedback for an indefinite stochastic linear-quadratic (LQ) problem. The coefficients are random, and the control may enter both the Brownian and Poisson terms.

Fix $T>0$ and integers $n,m\ge1$. The spaces $\mathbb R^d$, $\mathbb R^{d\times k}$ and $\mathbb S^d$ consist of real column vectors, real $d\times k$ matrices and symmetric $d\times d$ matrices, respectively. We use $I_d$ for the identity, ${}^\top$ for transpose, $|\cdot|$ for the Euclidean or Frobenius norm, and $K\succeq L$ when $K-L$ is positive semidefinite.

Let $(\Omega,\mathcal F,\mathbb F,\mathbb P)$ be a complete filtered probability space. The filtration $\mathbb F=(\mathcal F_s)_{0\le s\le T}$ is the usual augmentation of that generated by $W$ and $\mu$. Here $W$ is a one-dimensional Brownian motion, and $\mu$ is an independent Poisson random measure. The mark space $\mathbb Z$ has Borel sigma-algebra $\mathscr B(\mathbb Z)$, and the compensator of $\mu$ is $ds\,\nu(de)$, with $\nu(\mathbb Z)<\infty$. Set $\widetilde\mu(ds,de)=\mu(ds,de)-ds\,\nu(de)$, and let $\mathscr P$ be the predictable sigma-algebra. Every square-integrable martingale has a unique Brownian--Poisson representation \cite{HeWangYan1992}.

Write $\mathbb E$ for expectation under $\mathbb P$. For $1\le p<\infty$, $L^p(\mathcal F_t;\mathbb R^d)$ denotes the space of $\mathcal F_t$-measurable random vectors $\eta$ with $\mathbb E|\eta|^p<\infty$; the same notation applies at stopping times. The notation $L^\infty$ means essential boundedness. For a c\`adl\`ag (right-continuous with left limits) process, $Y_s=Y(s)$, $Y_{s-}=Y(s-)$ and $\Delta Y_s=Y_s-Y_{s-}$, with $Y_{0-}=Y_0$. We abbreviate almost surely and almost everywhere as a.s. and a.e., respectively. We suppress $\omega$.

\Needspace{8\baselineskip}
For $0\le t\le T$, consider the controlled linear stochastic differential equation with Poisson jumps
\begin{equation}\label{eq:1.1}
	\left\{
	\begin{aligned}
		dX(s) &= \bigl[A(s)X(s-)+B(s)u(s)\bigr]\,ds + \bigl[C(s)X(s-)+D(s)u(s)\bigr]\,dW(s) \\
		&\qquad + \int_{\mathbb{Z}}\bigl[E(s,e)X(s-)+F(s,e)u(s)\bigr]\,\widetilde\mu(ds,de), \qquad s\in(t,T],\\[2pt]
		X(t) &= \xi,
	\end{aligned}\right.
\end{equation}
Here $A(s),C(s),E(s,e)\in\mathbb R^{n\times n}$ and $B(s),D(s),F(s,e)\in\mathbb R^{n\times m}$. The initial data and admissible controls belong to
\[
 \begin{aligned}
 \mathcal D&=\{(t,\xi):0\le t\le T,\ \xi\in L^2(\mathcal F_t;\mathbb R^n)\},\\
 \mathcal U[t,T]&=\left\{u:(t,T]\times\Omega\to\mathbb R^m:\ u\text{ is predictable},\
             \mathbb E\int_t^T|u(s)|^2ds<\infty\right\}.
 \end{aligned}
\]
Controls are identified $ds\otimes d\mathbb P$-a.e. We denote the solution by $X^{t,\xi;u}$ and often write $X$ when the initial data and control are fixed.

\begin{ass}\label{ass:AS1}
The coefficients $A(\cdot),B(\cdot),C(\cdot),D(\cdot)$ are predictable and uniformly bounded $ds\otimes d\mathbb P$-a.e. The coefficients $E(\cdot,\cdot),F(\cdot,\cdot)$ are $\mathscr P\otimes\mathscr B(\mathbb Z)$-measurable and uniformly bounded $ds\otimes\nu(de)\otimes d\mathbb P$-a.e.
\end{ass}

{
\begin{ass}\label{ass:AS2}
The weighting matrices $Q:[0,T]\times\Omega\to\mathbb S^n$, $R:[0,T]\times\Omega\to\mathbb S^m$, $S:[0,T]\times\Omega\to\mathbb R^{m\times n}$ are uniformly bounded and $\mathbb{F}$-predictable. The terminal weight $G:\Omega\to\mathbb S^n$ is essentially bounded and $\mathcal{F}_T$-measurable.
\end{ass}
}

For $(t,\xi)\in\mathcal D$ and $u\in\mathcal U[t,T]$, the state equation has a unique square-integrable c\`adl\`ag solution. Define
\begin{align}
 \ell(s,x,u)&=x^\top Q(s)x+2u^\top S(s)x+u^\top R(s)u,\label{eq:running-cost}\\
 J(t,\xi;u)&=\mathbb E\!\left[X(T)^\top GX(T)+\int_t^T\ell(s,X(s),u(s))ds\,\middle|\,\mathcal F_t\right].\label{eq:cost-def}
\end{align}
\textbf{Stochastic linear-quadratic (SLQ) control problem.} Minimize $J(t,\xi;u)$ over $\mathcal U[t,T]$.
Its conditional essential infimum is denoted by $V(t,\xi)$. When attained, the optimal state and control are denoted by $(X^{t,\xi},u^{t,\xi})$, or $(X^*,u^*)$ for fixed initial data.

{\paragraph{The indefinite setting}
The symmetric weights $G,Q,R$ are not required to be positive semidefinite. The cross coefficient $S\in\mathbb R^{m\times n}$ may be nonzero. Negative state weights can coexist with a uniformly convex control-to-cost map, and the value may be negative. We impose uniform convexity \UC{} on the zero-initial-state cost in Definition~\ref{def:uniform_convexity}. Proposition~\ref{prop:quadratic} then gives a unique open-loop minimizer for each initial pair. The main task is to construct the value kernel and derive an admissible feedback from its Riccati equation.}

\subsection*{Related literature}
\paragraph{Classical theory and random coefficients}
Wonham \cite{Wonham1968} established the Riccati approach to stochastic LQ control for Brownian-driven systems. Bismut \cite{Bismut1976} developed the random-coefficient formulation, in which the Riccati equation is a backward stochastic equation rather than a deterministic matrix equation. Peng \cite{Peng1992} studied stochastic Hamilton--Jacobi--Bellman equations, while El Karoui, Peng and Quenez \cite{ElKaroui1997} developed the backward-equation framework and its financial applications. Under nonnegative state weights and a uniformly positive control weight, Tang \cite{Tang2003} proved solvability of the matrix-valued stochastic Riccati equation through the stochastic Hamiltonian system and the inverse of the optimal-state matrix. Tang \cite{Tang2015} later gave a dynamic-programming proof under the same definiteness assumptions. This proof combines the inverse of the zero-control fundamental matrix with Doob--Meyer decomposition to obtain the value kernel, and then uses localized completion of squares.

\paragraph{Indefinite weights and uniform convexity}
Chen, Li and Zhou \cite{Chen1998} showed that stochastic LQ problems can remain well posed even when the control weight is not nonnegative. Thus the classical positive-semidefiniteness requirements on the cost weights need not be necessary for solvability. Hu and Zhou \cite{Hu2003} proved uniqueness and established existence for several special cases of indefinite stochastic Riccati equations. Qian and Zhou \cite{Qian2013} considered a class driven by one-dimensional Brownian motion with an invertible control coefficient in the diffusion term. They used backward stochastic equations for the effective control-weight matrix and its inverse to enforce the positivity constraint. Du \cite{Du2015} developed sufficient solvability criteria through subsolutions. Sun, Li and Yong \cite{SunLiYong2016} distinguished open-loop and closed-loop solvability and characterized their relation to Riccati equations in the deterministic-coefficient setting. For random coefficients without jumps, Sun, Xiong and Yong \cite{SunXiongYong2021} showed that uniform convexity yields unique open-loop solvability, solvability of the associated stochastic Riccati equation and a feedback representation. Moon and Duncan \cite{MoonDuncan2020} obtained a direct completion-of-squares characterization under a positivity condition involving the Riccati solution. These results distinguish a condition on the original cost functional from a condition imposed on a candidate Riccati solution.

\paragraph{Jump-diffusion systems}
Poisson terms introduce discontinuous state changes and an additional martingale integrand in the Riccati equation. Maximum-principle and semimartingale approaches to control with jumps are studied in \cite{TangLi1994,Boel1980}; background on stochastic equations with jumps can be found in \cite{Situ2005}. For random-coefficient LQ systems, Zhang, Dong and Meng \cite{ZhangDongMeng2020} established solvability of the backward stochastic Riccati equation with jumps and constructed optimal feedback under nonnegative state weights and a uniformly positive control weight. Their proof combines dynamic programming, Doob--Meyer decomposition and state matrices on intervals between successive Poisson times to construct the value semimartingale and establish the Riccati equation.

Equivalent-cost and compensator methods provide another approach to indefinite LQ problems. Yu \cite{Yu2013} introduced equivalent cost functionals to establish solvability of stochastic Hamiltonian systems with random coefficients. Li, Wu and Yu \cite{LiWuYu2018} developed the relaxed-compensator method for jump systems and obtained open-loop optimal controls. They also established Riccati-based feedback for a class with control-independent jump coefficients ($F=0$) and additional measurability assumptions on the coefficients. Tang, Li and Wang \cite{TangLiWang2026} further applied equivalent costs to jump-diffusion systems with Brownian-adapted coefficients and obtained Riccati-based state feedback. These works connect indefinite problems with problems having nonnegative quadratic costs and uniformly positive definite control weights.

In this paper, we construct the Riccati solution from the conditional value field under uniform convexity of the original cost. The construction requires no prior choice of a matrix process that converts the cost to the standard positive-definite setting. We consider finite-activity jumps and bounded random coefficients. The state-equation and running-cost coefficients are predictable with respect to the joint filtration of $W$ and $\mu$, and the terminal weight $G$ is $\mathcal F_T$-measurable. Controls may enter both diffusion and jump terms. The Riccati equation includes an unknown Poisson martingale integrand, which also enters the effective control-weight matrix.

Moon and Chung \cite{MoonChung2021} characterized optimal controls by completion of squares for an affine jump-diffusion system. Their general feedback theorem assumes a solution of the Riccati equation and uniform positivity of the effective control weight. They also give a solvability criterion under nonnegative state weights, zero cross weight and a uniformly positive control weight. For the homogeneous finite-activity problem considered here, we instead construct the Riccati solution and establish its control-weight lower bound from uniform convexity. Li and Shi \cite{LiShi2022} characterized closed-loop solvability for Poisson-jump LQ problems with deterministic coefficients and indefinite weights through a deterministic Riccati equation. The present random-coefficient problem additionally requires construction of the Brownian and Poisson martingale integrands.

\paragraph{Related extensions and applications}
Regime switching introduces another source of randomness in the coefficients. Wen, Li, Xiong and Zhang \cite{WenLiXiongZhang2023} studied random-coefficient LQ control with Markovian switching and closed-loop representations under uniform convexity. Zhang, Li and Xiong \cite{ZhangLiXiong2021} analyzed open-loop and closed-loop solvability with state-equation coefficients and quadratic weights that are deterministic functions of time and regime. Wu, Li and Zhang \cite{WuLiZhang2025} considered an infinite horizon with regime-switching jumps. Control constraints lead to coupled backward equations: Hu, Shi and Xu \cite{HuShiXu2022} treated cone-constrained LQ control with regime switching, while Shi and Xu \cite{ShiXu2026} allowed controlled jump sizes. Related financial problems include investment--reinsurance under the Cram\'er--Lundberg model \cite{ShiXu2024} and liquidation with singular terminal values \cite{Fu2025}.

Riccati methods also arise in minimax control and stochastic differential games. Ba\c{s}ar and Bernhard \cite{Basar2008} developed the dynamic-game viewpoint for $H^\infty$ control. Open-loop and closed-loop saddle points and infinite-horizon zero-sum games are studied in \cite{SunYong2014,SunYongZhang2016}, and affine-quadratic games with random coefficients in \cite{Moon2020_Nash}. Stackelberg formulations are treated in \cite{SunWangWen2023}, with jump-diffusion counterparts in \cite{LiXiongYu2021,Moon2021_Stackelberg,WuXiongZhang2024}. In infinite dimensions, L\"u \cite{Lu2019} studied the relation between stochastic Riccati equations and closed-loop solvability for stochastic evolution equations; jump-driven evolution systems are considered in \cite{WangTangMeng2025}. Portfolio models provide a further motivation, from continuous-time allocation \cite{Merton1971} to discontinuous-price markets \cite{Jeanblanc1990}. The portfolio problem in Section~\ref{rev:finance} illustrates the compatibility of a negative terminal weight with uniform convexity and identifies a parameter threshold.

The most directly related constructions therefore concern random-coefficient Riccati equations, indefinite costs under uniform convexity, and feedback for jump systems. The next subsection states our results and explains how these strands enter the proof. It also compares the present construction with the contemporaneous work of Ding, Wen, Xiong and Zhang \cite{DingWenXiongZhang2026}.

\subsection*{Main results and method}\label{rev:intro}
Let Assumptions~\ref{ass:AS1}--\ref{ass:AS2} and \UC{} hold, with constant $\delta>0$. Theorems~\ref{thm:SREJ-existence-uniqueness} and \ref{thm:verification} yield a bounded symmetric c\`adl\`ag value kernel $P$. Its Brownian and compensated-Poisson integrands are symmetric predictable processes $\Lambda(s),\Xi(s,e)\in\mathbb S^n$. For this triple, set
\begin{equation}\label{eq:introN}
 \begin{aligned}
 \mathscr N(s)={}&R(s)+D(s)^\top P(s-)D(s)
       +\int_{\mathbb Z}F(s,e)^\top[P(s-)+\Xi(s,e)]F(s,e)\nu(de),\\
 \mathscr M(s)={}&P(s-)B(s)+\Lambda(s)D(s)+C(s)^\top P(s-)D(s)+S(s)^\top\\
 &+\int_{\mathbb Z}\!\left[E(s,e)^\top P(s-)F(s,e)
        +(I_n+E(s,e))^\top\Xi(s,e)F(s,e)\right]\nu(de).
 \end{aligned}
\end{equation}
\Needspace{5\baselineskip}
Here $\mathscr N(s)\in\mathbb S^m$ is the effective control-weight matrix and $\mathscr M(s)\in\mathbb R^{n\times m}$ is the cross coefficient. For every stopping time $\tau\le T$ and $\xi\in L^2(\mathcal F_\tau;\mathbb R^n)$, the optimal pair satisfies
\begin{equation}\label{eq:intro-result}
 V(\tau,\xi)=\xi^\top P(\tau)\xi,\qquad
 \mathscr N\succeq\delta I_m,\qquad
 u^{\tau,\xi}=-\mathscr N^{-1}\mathscr M^\top X^{\tau,\xi}_- .
\end{equation}
The matrix inequality and feedback identity hold $ds\otimes d\mathbb P$-a.e. on the relevant intervals. We suppress time arguments in pointwise matrix identities, with $E,F,\Xi$ evaluated at the same time and mark.

The global inverse-matrix construction of Sun, Xiong and Yong \cite{SunXiongYong2021} does not extend directly to the present jump setting. They first prove global invertibility of the optimal-state matrix $\mathbf X$ and then recover the Riccati kernel from $P=\mathbf Y\mathbf X^{-1}$, where $\mathbf Y$ is the adjoint matrix. Their stopping-time contradiction uses left continuity of the value kernels to pass from equality before the first singular time to equality at that time. In the present jump setting, however, the value kernels are generally only c\`adl\`ag. At a possible jump time $\tau$, equality before $\tau$ therefore gives only
\begin{equation}\label{eq:left-limit-obstruction}
 P_t=\overline P_t\ (t<\tau)
 \ \Longrightarrow\ P_{\tau-}=\overline P_{\tau-},
 \qquad
 P_{\tau}=\overline P_{\tau}\ \text{need not hold}.
\end{equation}
Modifying this limiting argument would not remove a second obstruction: the optimal-state matrix can actually lose rank at a jump. Define the feedback gain $\mathscr Q(s)=\mathscr N(s)^{-1}\mathscr M(s)^\top\in\mathbb R^{m\times n}$. At a Poisson atom $(s,e)$,
\begin{equation}\label{eq:closed-jump-multiplier}
 X^*(s)=[I_n+E(s,e)-F(s,e)\mathscr Q(s)]X^*(s-).
\end{equation}
The matrices $I_n+E$ and $I_n+E-F\mathscr Q$ are the jump multipliers of the zero-control and optimal closed-loop systems, respectively. Example~\ref{ex:rankloss} satisfies \UC{} with $I_n+E=1$ and $I_n+E-F\mathscr Q=0$. Thus neither \UC{} nor nonsingularity of $I_n+E$ guarantees global invertibility of the optimal-state matrix. A global inverse-matrix construction would have to exclude such singular closed-loop multipliers, a restriction involving the unknown optimal gain. This rank-loss obstruction also applies to the optimal-state inverse in Tang's Hamiltonian construction \cite{Tang2003}. Such global inverse-matrix constructions therefore cannot cover the full class considered here.

Our construction builds on the dynamic programming and semimartingale methods of Tang \cite{Tang2015} and Zhang, Dong and Meng \cite{ZhangDongMeng2020}. In their models, the state and terminal weights are positive semidefinite, the control weight is uniformly positive definite, and there is no state--control cross term. The cost is therefore nonnegative, which gives $P\succeq0$. In the jump case, Zhang, Dong and Meng also use Poisson compensation to prove that the jump contribution to the effective control weight is positive semidefinite. Written in our notation, these properties give
\[
 \mathscr N
 =R+D^\top P_-D+\int_{\mathbb Z}F^\top(P_-+\Xi)F\,\nu(de)
 \succeq R\succeq cI_m,\qquad c>0.
\]
This bound allows minimization of the quadratic drift over the control and makes the inverse in the Riccati generator well defined.

For indefinite weights, $P$ and $P_-+\Xi$ need not be positive semidefinite, and $R$ need not have a positive lower bound. Consequently, the estimate $\mathscr N\succeq R$ cannot be inferred from our assumptions. The value-field construction remains useful, but this step in the derivation of the Riccati equation requires a different argument.

We first construct a bounded value semimartingale with a general predictable finite-variation part in Theorem~\ref{thm:semimartingale_P}, using continuous fundamental matrices between jumps and endpoint corrections. Lemma~\ref{lem:residual-measure} obtains $\mathscr N\succeq0$ from the quadratic drift inequalities for all state--control directions. The same construction applies to the auxiliary problem with weight $R-\varepsilon I_m$, which remains uniformly convex for $0<\varepsilon<\delta$. Value comparison and Poisson compensation then yield the required order between the two post-jump value kernels and give $\mathscr N\succeq\varepsilon I_m$. Letting $\varepsilon\uparrow\delta$ provides $\mathscr N\succeq\delta I_m$. With this lower bound established, optimality identifies the entire predictable finite-variation part as the Riccati drift.

Once the Riccati drift is identified, Theorem~\ref{thm:SREJ-existence-uniqueness} gives conditional energy bounds at arbitrary stopping times and finite-order conditional moments of the integrated energies and drift variation. Theorem~\ref{thm:verification} identifies the feedback with the open-loop optimizer, so it inherits the conditional energy estimate. Both $I_n+E$ and $I_n+E-F\mathscr Q$ may be singular. The value solution is maximal among bounded strongly regular solutions. It is unique in the feedback-admissible class.

\label{loc:related-work}In a contemporaneous study, Ding, Wen, Xiong and Zhang \cite{DingWenXiongZhang2026} obtain closely related Riccati and feedback results. They assume uniform convexity and finite jump intensity, and also allow singular jump multipliers.\footnote{The first version of this paper, \href{https://arxiv.org/abs/2605.12775v1}{arXiv:2605.12775v1}, was submitted to arXiv on 12 May 2026. The first version of \cite{DingWenXiongZhang2026} was submitted on 13 May 2026. The comparison here concerns their third version, dated 23 October 2026.} Their control-weight lower bound is proved using localized spike controls and vanishing short-interval remainders \cite[Lemmas~5.5--5.6]{DingWenXiongZhang2026}. Step~2 of Lemma~\ref{lem:residual-measure} instead establishes nonnegativity of the effective control weight by evaluating the quadratic drift for all state--control directions. The same construction and argument apply to the auxiliary problem with weight $R-\varepsilon I_m$, $0<\varepsilon<\delta$, before its Riccati drift is identified. Step~3 then combines value comparison with Poisson compensation. With hats denoting the auxiliary value quantities, this gives
\[
 \begin{gathered}
 K=(P_-+\Xi)-(\widehat P_-+\widehat\Xi)\succeq0,\\
 \mathscr N-\widehat{\mathscr N}
 =\varepsilon I_m+D^\top(P_--\widehat P_-)D
       +\int_{\mathbb Z}F^\top K F\,\nu(de)
 \succeq\varepsilon I_m.
 \end{gathered}
\]
The order of $K$ holds under $ds\otimes\nu(de)\otimes d\mathbb P$. Together with $\widehat{\mathscr N}\succeq0$, this yields $\mathscr N\succeq\delta I_m$ as $\varepsilon\uparrow\delta$. This value comparison replaces the fourth-moment decomposition and short-interval remainder estimates for spike states. Both approaches subsequently identify the Riccati drift through optimality and state matrices on jump intervals.

\label{end:related-work}

\subsection*{Organization}
Sections~\ref{sec:prelim}--\ref{sec:DPP_semimartingale} construct the value kernel. Sections~\ref{sec:SREJ}--\ref{sec:feedback} establish the Riccati equation and feedback. Section~\ref{sec:application} gives examples, and Appendix~\ref{app:product} records the expansion of It\^o's formula.

\section{Preliminaries and Problem Formulation}\label{sec:prelim}

For matrices, $|K|^2=\operatorname{tr}(K^\top K)$, with $\operatorname{tr}$ the trace, and $\|K\|_{\rm op}=\sup_{|x|=1}|Kx|$. We use $a^\pm=\max\{\pm a,0\}$, $a\wedge b=\min\{a,b\}$, $a\vee b=\max\{a,b\}$, $\inf\varnothing=+\infty$, and $\mathbf1_H$ for the indicator of $H$. Constants denoted by $C,C_j$ are deterministic and may vary between estimates.

Let $\mathcal T$ be the $[0,T]$-valued $\mathbb F$-stopping times, and set $\mathbb E_\tau[\cdot]=\mathbb E[\cdot\mid\mathcal F_\tau]$. For $\tau\in\mathcal T$, let $\mathcal U[\tau,T]$ be the space from Section~1 on $(\tau,T]$. Extend its controls by zero before $\tau$. Its norm is $\|u\|_{\mathcal U}^2=\|u\|_{\mathcal U[\tau,T]}^2=\mathbb E\int_\tau^T|u_s|^2ds$. For an initial pair $(\tau,\xi)$, $X^{\tau,\xi;u}$ solves \eqref{eq:1.1} from $X_\tau=\xi$, with the original coefficients and driving processes. The initial value is taken after any jump at $\tau$, and the Poisson integral is over $(\tau,s]$. The cost $J(\tau,\xi;u)$ is extended to these initial pairs. We write $\mathcal S^2[\tau,T]$ for the adapted c\`adl\`ag $\mathbb R^n$-valued processes with $\mathbb E\sup_{\tau\le s\le T}|X_s|^2<\infty$. We use
$\|Y\|_\infty=\operatorname*{ess\,sup}_{\omega\in\Omega}\sup_{0\le s\le T}|Y_s(\omega)|$
for a c\`adl\`ag vector or matrix process $Y$. For a stopping time $\sigma$, we write $Y_t^\sigma=Y_{t\wedge\sigma}$.

\Needspace{5\baselineskip}
\subsection{State equation estimates}
The next estimates are uniform in the initial stopping time. They will be used for conditional minimization and for concatenating a control with an optimal continuation.
\begin{lem}\label{lem:state}
Under Assumption~\ref{ass:AS1}, for every $\tau\in\mathcal T$, $\xi\in L^2(\mathcal F_\tau;\mathbb R^n)$, and $u\in\mathcal U[\tau,T]$, the state equation with initial pair $(\tau,\xi)$ has a unique solution in $\mathcal S^2[\tau,T]$. There is a deterministic constant $C$, independent of $\tau,\xi,u$, such that
\begin{equation}\label{eq:state-estimate}
 \mathbb E_\tau\sup_{\tau\le s\le T}|X(s)|^2
 \le C\left(|\xi|^2+\mathbb E_\tau\int_\tau^T|u(s)|^2\,ds\right).
\end{equation}
Uniqueness is up to indistinguishability. For any $\xi'\in L^2(\mathcal F_\tau;\mathbb R^n)$ and $u'\in\mathcal U[\tau,T]$, the same constant $C$ satisfies
\[
 \mathbb E_\tau\sup_{\tau\le s\le T}
 |X^{\tau,\xi;u}(s)-X^{\tau,\xi';u'}(s)|^2
 \le C\left(|\xi-\xi'|^2+
 \mathbb E_\tau\int_\tau^T|u(s)-u'(s)|^2\,ds\right).
\]
Moreover, for every $\sigma\in\mathcal T$ with $\tau\le\sigma$,
\begin{equation}\label{eq:restarting}
 X(s)=X^{\sigma,X(\sigma);\,u|_{(\sigma,T]}}(s),\qquad \sigma\le s\le T,
\end{equation}
up to indistinguishability.
\end{lem}

\begin{proof}
For $\sigma\in\mathcal T$ with $\tau\le\sigma$, set $\theta=(\sigma+h)\wedge T$. Let $\mathfrak P$ be the integral map in \eqref{eq:1.1} on $[\sigma,\theta]$, with fixed initial state and control. The two It\^o isometries and Doob's inequality give, for adapted c\`adl\`ag $Y,Z$ with finite second supremum moments,
\[
 \mathbb E\sup_{\sigma\le s\le\theta}|(\mathfrak PY)_s-(\mathfrak PZ)_s|^2
 \le C(h+h^2)\mathbb E\sup_{\sigma\le s\le\theta}|Y_s-Z_s|^2.
\]
Choose $h>0$ with $C(h+h^2)<1$. Iteration and Gronwall's inequality yield existence, uniqueness and the unconditional estimate on $[\tau,T]$; see also \cite{Galchuk1978,Situ2005,Protter2005,ZhangDongMeng2020}.

For $H\in\mathcal F_\tau$, use initial data $\mathbf1_H\xi$ and control $\mathbf1_Hu$ on $(\tau,T]$. The multiplier $\mathbf1_H\mathbf1_{(\tau,T]}$ is predictable, so integral locality and uniqueness identify the new state with $\mathbf1_HX$ on $[\tau,T]$. The unconditional estimate therefore gives
\[
 \mathbb E\!\left[\mathbf1_H\sup_{\tau\le s\le T}|X_s|^2\right]
 \le C\mathbb E\!\left[\mathbf1_H\left(|\xi|^2+\int_\tau^T|u_s|^2ds\right)\right].
\]
Since $H\in\mathcal F_\tau$ is arbitrary, the defining property of conditional expectation gives \eqref{eq:state-estimate}. Linearity gives the difference estimate. Split all integrals at $\sigma$. The restriction of $X$ to $[\sigma,T]$ and the solution started from $(\sigma,X(\sigma))$ solve the same equation on $(\sigma,T]$. Pathwise uniqueness proves \eqref{eq:restarting}. In particular, a jump at $\sigma$ belongs to the first interval only.
\label{end:lem:state}
\end{proof}

\begin{rmk}\label{rmk:moments}
Under Assumption~\ref{ass:AS1}, zero-control states with $L^p$ initial data have $p$th supremum moments for $p\ge2$. Controlled jump equations additionally require suitable control moments, for example
$\mathbb E(\int_\tau^T|u|^2ds)^{p/2}+\mathbb E\int_\tau^T|u|^pds<\infty$.
Only $p=2$ is used for general admissible controls here.
\end{rmk}

\subsection{Linear matrix equations between Poisson times}\label{sec:removed-assumption}
The value-kernel construction and the drift analysis use invertible matrix solutions between successive Poisson times. The equations below retain the compensation drift but have no Poisson integral. Their solutions are continued through the next Poisson time without a jump update; see \cite{Kunita1990} for the underlying linear stochastic-flow theory.

\begin{lem}\label{lem:restarted-reference}
Let Assumption~\ref{ass:AS1} hold. Set $\tau_0=0$. Let $\tau_1,\tau_2,\ldots$ be the successive Poisson times. Assign $+\infty$ when no subsequent time exists. For $\tau_k<T$, put $b_k=\tau_{k+1}\wedge T$ and
\begin{equation}\label{eq:interjump-coefficients}
 A^0=A-\int_{\mathbb Z}E(\cdot,e)\nu(de),\qquad
 B^0=B-\int_{\mathbb Z}F(\cdot,e)\nu(de).
\end{equation}
Let $\Phi^k$ be the fundamental matrix of the following homogeneous linear stochastic equation, normalized by $\Phi^k(\tau_k)=I_n$:
\begin{equation}\label{eq:continuous-reference}
 d\Phi^k=A^0\Phi^k ds+C\Phi^k dW,\qquad
 \Phi^k(\tau_k)=I_n.
\end{equation}
Its continuous inverse $\Psi^k$ exists on $[\tau_k,T]$ and satisfies
\begin{equation}\label{eq:continuous-inverse}
 d\Psi^k=\Psi^k(-A^0+CC)ds-\Psi^k C\,dW,\qquad
 \Psi^k(\tau_k)=I_n.
\end{equation}
If $X^{k,x;u}$ is the jump-system state started from $(\tau_k,x)$ under control $u$, then on $[\tau_k,b_k)$,
\begin{align}\label{eq:interjump-variation}
 X^{k,x;u}(s)=\Phi^k(s)\bigg[x
 &+\int_{\tau_k}^s\Psi^k(r)(B^0(r)-C(r)D(r))u(r)dr\notag\\
 &+\int_{\tau_k}^s\Psi^k(r)D(r)u(r)dW(r)\bigg].
\end{align}
The left-limit version holds for $s\in(\tau_k,b_k]$. In particular, for fixed $u$ the initial-state map is an affine homeomorphism before the next jump and at its left limit.
\end{lem}
\begin{proof}
Between atoms, the compensation term is $-\int(EX_-+Fu)\nu(de)ds$, giving the bounded coefficients $A^0,B^0,C,D$. Solve \eqref{eq:continuous-reference} and \eqref{eq:continuous-inverse} for $\Phi^k$ and $\Psi^k$ on $[\tau_k,T]$. Both solutions have continuous paths, including at $b_k$. Applying It\^o's formula to $\Psi_s^k\Phi_s^k$ gives
\[
 \begin{aligned}
 d(\Psi_s^k\Phi_s^k)
 &=\Psi_s^k[(-A_s^0+C_sC_s)+A_s^0-C_sC_s]\Phi_s^k\,ds\\
 &\quad+(-\Psi_s^kC_s\Phi_s^k+\Psi_s^kC_s\Phi_s^k)\,dW_s=0.
 \end{aligned}
\]
Thus $\Psi^k\Phi^k=I_n$; the inverse is two-sided, and bounded coefficients give all finite supremum moments.

Write $Y_s=X_s^{k,x;u}$ on $[\tau_k,b_k)$. Applying It\^o's formula to $\Psi_s^kY_s$ yields
\[
 \begin{aligned}
 d(\Psi_s^kY_s)
 &=\Psi_s^k[(-A_s^0+C_sC_s)Y_s+A_s^0Y_s+B_s^0u_s
   -C_s(C_sY_s+D_su_s)]\,ds\\
 &\quad+\Psi_s^k[-C_sY_s+C_sY_s+D_su_s]\,dW_s\\
 &=\Psi_s^k(B_s^0-C_sD_s)u_s\,ds
   +\Psi_s^kD_su_s\,dW_s.
 \end{aligned}
\]
Integration gives \eqref{eq:interjump-variation}; continuity of $\Phi^k$ and $\Psi^k$ extends it to $b_k-$. At an atom $(b_k,e)$, the actual state instead satisfies
\begin{equation}\label{eq:actual-reference-endpoint}
 X^{k,x;u}(b_k)=(I_n+E(b_k,e))X^{k,x;u}(b_k-)+F(b_k,e)u(b_k).
\end{equation}
Before this update and at its left limit,
$X^{k,x;u}=\Phi^k x+X^{k,0;u}$, proving the affine-homeomorphism assertion. The finitely many coordinate solutions determine consistent versions for all $x$ for a fixed $u$; the countably many $k$ and controls used below may then share one exceptional set.
\label{end:lem:restarted-reference}
\end{proof}

\subsection{Uniform convexity}
We impose coercivity on the control-to-cost map. This condition gives conditional open-loop minimizers and, subsequently, a uniform lower bound for the effective control-weight matrix.
\begin{defn}\label{def:uniform_convexity}
Condition \UC{} means that there is a deterministic $\delta>0$ such that, for every deterministic $t\in[0,T]$ and $u\in\mathcal U[t,T]$,
\begin{equation}\label{eq:UC}
 J(t,0;u)\ge\delta\,\mathbb E_t\int_t^T|u(s)|^2\,ds
 \qquad\text{a.s.}
\end{equation}
The same $\delta$ is used for all initial times.
\end{defn}

\begin{rmk}\label{rmk:UC_sufficient}
If $G,Q\succeq0$, $S=0$ and $R\succeq\delta I_m$, then \UC{} holds. Another sufficient condition is
\[
 G\succeq gI_n,\quad Q,R\succeq0,\quad S=0,\qquad
 D^\top D+\int_{\mathbb Z}F^\top F\,\nu(de)\succeq cI_m,
 \quad g,c>0.
\]
Indeed, for $X_\tau=0$ set $L=B+C^\top D+\int E^\top F\,\nu(de)$. The drift of $|X|^2$ is
\[
 \begin{aligned}
 &X_-^\top\!\left(A+A^\top+C^\top C+\int E^\top E\,\nu(de)\right)X_-
 +2X_-^\top Lu\\
 &\qquad+u^\top\!\left(D^\top D+\int F^\top F\,\nu(de)\right)u
 \ge \tfrac c2|u|^2-K|X_-|^2,
 \end{aligned}
\]
for a deterministic $K\ge0$. It\^o's formula for $e^{K(s-\tau)}|X_s|^2$, followed by conditional expectation, gives
\[
 J(\tau,0;u)\ge g\mathbb E_\tau|X_T|^2
 \ge \tfrac{gc}{2}e^{-KT}\mathbb E_\tau\int_\tau^T|u_s|^2ds.
\]

\end{rmk}

\section{Quadratic Structure and Semimartingale Property of the Value Process}\label{sec:DPP_semimartingale}

Uniform convexity first yields conditional minimizers for every stopping-time initial pair. Their energy estimate justifies optimal continuation in the dynamic programming principle (DPP), which is then used to construct the semimartingale value kernel.

\subsection{Conditional solvability and the value kernel}
\begin{prop}\label{prop:quadratic}
Under Assumptions~\ref{ass:AS1}--\ref{ass:AS2} and \UC{},
for every $\tau\in\mathcal T$ and
$\xi\in L^2(\mathcal F_\tau;\mathbb R^n)$, there is a
unique open-loop optimal pair
$(X^{\tau,\xi},u^{\tau,\xi})$.
Moreover, \eqref{eq:UC} extends to stopping times:
\begin{equation}\label{eq:UC-stopping}
 J(\tau,0;u)
 \ge\delta\,\mathbb E_\tau\int_\tau^T|u(s)|^2\,ds,
 \qquad u\in\mathcal U[\tau,T].
\end{equation}
There is a bounded symmetric
$\mathcal F_\tau$-measurable matrix $P(\tau)$ such that
\begin{equation}\label{eq:value-quadratic}
 V(\tau,\xi)=\xi^\top P(\tau)\xi,
 \qquad |P(\tau)|\le C.
\end{equation}
The optimizer depends linearly and
$\mathcal F_\tau$-locally on $\xi$, and
\begin{equation}\label{eq:optimal-estimate}
 \mathbb E_\tau\left[
   \int_\tau^T|u^{\tau,\xi}(s)|^2\,ds
   +\sup_{\tau\le s\le T}|X^{\tau,\xi}(s)|^2
 \right]\le C|\xi|^2.
\end{equation}
All constants are deterministic and independent of
$\tau$ and $\xi$.
\end{prop}
\begin{proof}
For $H\in\mathcal F_\tau$ and $u\in\mathcal U[\tau,T]$, put
$\widehat u=\mathbf1_H\mathbf1_{(\tau,T]}u$. The zero-initial-state solution under $\widehat u$ vanishes before $\tau$ and equals $\mathbf1_HX^{\tau,0;u}$ afterwards. Thus \UC{} at time zero gives
\[
 \mathbb E[\mathbf1_HJ(\tau,0;u)]
 =\mathbb EJ(0,0;\widehat u)
 \ge\delta\mathbb E\!\left[\mathbf1_H\int_\tau^T|u|^2ds\right].
\]
Since this inequality holds for every $H\in\mathcal F_\tau$, it implies the conditional inequality \eqref{eq:UC-stopping}.

Write $Y=X^{\tau,\xi;0}$ and $Z_u=X^{\tau,0;u}$. Define
\begin{align*}
 \mathcal A_\tau(u,v)&=\mathbb E_\tau\!\left[Z_u(T)^\top GZ_v(T)+
 \int_\tau^T(Z_u^\top QZ_v+u^\top SZ_v+v^\top SZ_u+u^\top Rv)ds\right],\\
 \mathcal L_{\tau,\xi}(v)&=\mathbb E_\tau\!\left[Y(T)^\top GZ_v(T)+
 \int_\tau^T(Y^\top QZ_v+v^\top SY)ds\right],\qquad
 \mathcal C_{\tau,\xi}=J(\tau,\xi;0).
\end{align*}
Linearity and \eqref{eq:state-estimate} yield
\begin{equation}\label{eq:conditional-quadratic-decomposition}
 J(\tau,\xi;u)=\mathcal A_\tau(u,u)+2\mathcal L_{\tau,\xi}(u)+\mathcal C_{\tau,\xi},
 \qquad \mathcal A_\tau(u,u)=J(\tau,0;u),
\end{equation}
and the following bounds. Write $\|u\|_\tau^2=\mathbb E_\tau\int_\tau^T|u|^2ds$ for the $\mathcal F_\tau$-measurable conditional energy. Its expectation is the squared Hilbert-space norm: $\|u\|_{\mathcal U[\tau,T]}^2=\mathbb E\|u\|_\tau^2$. Then
\begin{equation}\label{eq:conditional-bilinear-bound}
 |\mathcal A_\tau(u,v)|\le C\|u\|_\tau\|v\|_\tau,\qquad
 |\mathcal L_{\tau,\xi}(v)|\le C|\xi|\|v\|_\tau,\qquad
 |\mathcal C_{\tau,\xi}|\le C|\xi|^2.
\end{equation}
The real-valued bilinear form $a_\tau(u,v)=\mathbb E\mathcal A_\tau(u,v)$ on the Hilbert space $\mathcal U[\tau,T]$ is bounded and satisfies
$a_\tau(u,u)\ge\delta\|u\|_{\mathcal U}^2$.
Lax--Milgram \cite{Yosida1980} gives a unique $u^*$ such that
\begin{equation}\label{eq:variational}
 a_\tau(u^*,v)+\mathbb E\mathcal L_{\tau,\xi}(v)=0,
 \qquad v\in\mathcal U[\tau,T].
\end{equation}
Since $Z_{\mathbf1_Hv}=\mathbf1_HZ_v$, substitute $\mathbf1_Hv$ for $v$ in \eqref{eq:variational}. As $H\in\mathcal F_\tau$ is arbitrary, we obtain
\begin{equation}\label{eq:conditional-stationarity}
 \mathcal A_\tau(u^*,v)+\mathcal L_{\tau,\xi}(v)=0
 \quad\text{a.s. for each }v\in\mathcal U[\tau,T].
\end{equation}
For any competitor $u$, substitute $v=u-u^*$ and expand the cost:
\begin{align}\label{eq:conditional-gap}
 J(\tau,\xi;u)-J(\tau,\xi;u^*)
 &=\mathcal A_\tau(u-u^*,u-u^*)\notag\\
 &\ge\delta\mathbb E_\tau\int_\tau^T|u-u^*|^2ds.
\end{align}
This proves conditional optimality and uniqueness. With $v=u^*$ in \eqref{eq:conditional-stationarity},
\[
 \delta\|u^*\|_\tau^2\le-\mathcal L_{\tau,\xi}(u^*)
 \le C|\xi|\|u^*\|_\tau.
\]
Together with \eqref{eq:state-estimate}, this proves \eqref{eq:optimal-estimate} and $|V(\tau,\xi)|\le C|\xi|^2$.

The variational equation and uniqueness give linearity in $\xi$ and
$u^{\tau,\mathbf1_H\xi}=\mathbf1_Hu^{\tau,\xi}$. Consequently $V(\tau,x)$ is quadratic for deterministic $x$. Let $e_1,\ldots,e_n$ be the coordinate vectors of $\mathbb R^n$ and define
\begin{equation}\label{eq:polarization}
 P_{ij}(\tau)=\tfrac14[V(\tau,e_i+e_j)-V(\tau,e_i-e_j)].
\end{equation}
Using the finitely many coordinate optimizers gives $V(\tau,x)=x^\top P(\tau)x$. A common null set for rational $x$, followed by continuity, yields $|P(\tau)|\le C$. For a simple state $\xi=\sum_{i=1}^r x_i\mathbf1_{H_i}$, where $(H_i)_{i=1}^r$ is a finite $\mathcal F_\tau$-measurable partition, locality gives
\[
 V(\tau,\xi)=\sum_{i=1}^r\mathbf1_{H_i}V(\tau,x_i)
            =\xi^\top P(\tau)\xi.
\] For $\xi_j\to\xi$ in $L^2$, linearity and \eqref{eq:optimal-estimate} imply
\[
 \mathbb E\int_\tau^T|u^{\tau,\xi_j}-u^{\tau,\xi}|^2ds
 \le C\mathbb E|\xi_j-\xi|^2\longrightarrow0.
\]
The state estimate and bounded cost coefficients then give
$V(\tau,\xi_j)\to V(\tau,\xi)$ in $L^1$, proving \eqref{eq:value-quadratic}.
\label{end:prop:quadratic}
\end{proof}

\subsection{Dynamic programming}

For $(\tau,\xi)$ and $u\in\mathcal U[\tau,T]$, write $X^u=X^{\tau,\xi;u}$. The family defined in \eqref{eq:cost-to-go} adds the running cost already incurred to the optimal continuation value. An integrable stopping-time family $Y(\theta)$ is consistent if $Y(\theta_1)=Y(\theta_2)$ a.s. on $\{\theta_1=\theta_2\}$. It is of class D if $\{Y(\theta):\theta\in\mathcal T\}$ is uniformly integrable.

\begin{thm}\label{thm:DPP}
	Under the assumptions of Proposition~\ref{prop:quadratic},
	for every $\tau\in\mathcal T$,
	$\xi\in L^2(\mathcal F_\tau;\mathbb R^n)$, and
	$\sigma\in\mathcal T$ satisfying $\tau\le\sigma\le T$,
	\begin{equation}\label{eq:DPP}
		V(\tau,\xi)
		=
		\operatorname*{ess\,inf}_{u\in\mathcal U[\tau,T]}
		\mathbb E_\tau\left[
		\int_\tau^\sigma
		\ell(s,X^u(s),u(s))\,ds
		+V(\sigma,X^u(\sigma))
		\right].
	\end{equation}
	For each fixed $u\in\mathcal U[\tau,T]$, the family
	\begin{equation}\label{eq:cost-to-go}
		\mathcal J^{\tau,\xi;u}(\theta)
		=
		V(\theta,X^u(\theta))
		+\int_\tau^\theta
		\ell(s,X^u(s),u(s))\,ds,
		\qquad
		\theta\in\mathcal T,\quad \tau\le\theta\le T,
	\end{equation}
	is a consistent class-D submartingale system.
	Along the optimal pair, it is a martingale system.
\end{thm}

\begin{proof}
Fix $(\tau,\xi)$ and write $X^u=X^{\tau,\xi;u}$. For $u\in\mathcal U[\tau,T]$ and $v\in\mathcal U[\sigma,T]$, the control
$w=u\mathbf1_{(\tau,\sigma]}+v\mathbf1_{(\sigma,T]}$ is predictable and square integrable. Identity \eqref{eq:restarting} and the tower property give
\begin{equation}\label{eq:DPP-concatenation}
 J(\tau,\xi;w)=\mathbb E_\tau\!\left[
 \int_\tau^\sigma\ell(s,X_s^u,u_s)ds+J(\sigma,X_\sigma^u;v)\right].
\end{equation}
For
$H(u)=\mathbb E_\tau[\int_\tau^\sigma\ell(s,X_s^u,u_s)ds+V(\sigma,X_\sigma^u)]$,
conditional optimality yields $J(\tau,\xi;u)\ge H(u)$, and hence
$V(\tau,\xi)\ge\operatorname*{ess\,inf}_uH(u)$.
Conversely, Proposition~\ref{prop:quadratic} supplies an optimal continuation $v^*=u^{\sigma,X_\sigma^u}$, with
\[
 \mathbb E\int_\sigma^T|v^*|^2ds\le C\mathbb E|X_\sigma^u|^2<\infty.
\]
Using $v^*$ in \eqref{eq:DPP-concatenation} gives $V(\tau,\xi)\le H(u)$ for every $u$, and proves \eqref{eq:DPP}. At a jump time $\sigma$, $X_\sigma^u$ is the post-jump initial state of the continuation.

For $\tau\le\theta_1\le\theta_2\le T$, the DPP implies
\begin{equation}\label{eq:submartingale-system}
 \mathcal J^{\tau,\xi;u}(\theta_1)
 \le\mathbb E_{\theta_1}\mathcal J^{\tau,\xi;u}(\theta_2).
\end{equation}
Locality of the state equation and conditional cost on $\{\theta_1=\theta_2\}$ gives consistency. Moreover,
\begin{equation}\label{eq:DPP-classD-bound}
 |\mathcal J^{\tau,\xi;u}(\theta)|
 \le C\left(\sup_{\tau\le s\le T}|X_s^u|^2+
                   \int_\tau^T|u_s|^2ds\right)=:D_u,
 \qquad\mathbb ED_u<\infty,
\end{equation}
which proves class D. For an optimal pair let $Z_\theta=\mathcal J^{\tau,\xi;u^*}(\theta)$. Then
\[
 Z_\tau=\mathbb E_\tau Z_T,\qquad
 Z_\tau\le\mathbb E_\tau Z_\theta\le\mathbb E_\tau Z_T,
 \qquad \mathbb E_\theta Z_T-Z_\theta\ge0.
\]
The last difference has expectation zero, so $Z_\theta=\mathbb E_\theta Z_T$. This proves the martingale assertion.
\label{end:thm:DPP}
\end{proof}

\subsection{Semimartingale aggregation}

The value matrices are initially defined separately at stopping times. Aggregation constructs a single c\`adl\`ag process whose value at each stopping time agrees a.s. with the corresponding matrix. We then obtain its canonical decomposition, keeping the predictable finite-variation term general.

\begin{thm}\label{thm:semimartingale_P}
	Under Assumptions~\ref{ass:AS1}--\ref{ass:AS2} and
	\UC{}, the family
	$\{P(\tau):\tau\in\mathcal T\}$ has a bounded adapted
	c\`adl\`ag aggregation, also denoted by $P$.
	Moreover, $P$ is a special semimartingale of the form
	\begin{equation}\label{eq:P-general-FV}
		P(t)
		=
		P(0)+\mathcal B(t)
		+\int_0^t\Lambda(s)\,dW(s)
		+\int_0^t\int_{\mathbb Z}
		\Xi(s,e)\,\widetilde\mu(ds,de),
		\qquad P(T)=G,
	\end{equation}
	where $\mathcal B(0)=0$ and $\mathcal B$ is a predictable
	symmetric finite-variation process. The predictable
	martingale integrands $\Lambda$ and $\Xi$ are symmetric,
	\[
	\int_0^T|\Lambda(s)|^2\,ds<\infty
	\quad\text{a.s.},
	\]
	and
	\begin{equation}\label{eq:Xi-bounded}
		|\Xi(s,e)|\le2\|P\|_\infty
		\quad
		ds\otimes\nu(de)\otimes d\mathbb P\text{-a.e.}
	\end{equation}
\end{thm}

\begin{proof}
Write $P_{\mathrm v}(\theta)$ for the stopping-time value matrix in Proposition~\ref{prop:quadratic}.

\textbf{Step 1. Aggregation along controlled states.}
Fix $(\tau,\xi;u)$ and write $X=X^{\tau,\xi;u}$ and
$\mathcal J(\theta)=V(\theta,X_\theta)+\int_\tau^\theta\ell(s,X_s,u_s)ds$.
Theorem~\ref{thm:DPP} gives a consistent class-D submartingale system, with
\begin{equation}\label{eq:aggregation-domination}
 |\mathcal J(\theta)|\le C\left(\sup_{\tau\le s\le T}|X_s|^2+
 \int_\tau^T|u_s|^2ds\right)=:D_u,\qquad \mathbb ED_u<\infty.
\end{equation}
Let $\sigma_j\downarrow\sigma\ge\tau$, and let $(Y,v)$ be optimal from $(\sigma,X_\sigma)$. The martingale assertion gives
\[
 V(\sigma,X_\sigma)=\mathbb E_\sigma\!\left[
 \int_\sigma^{\sigma_j}\ell(s,Y_s,v_s)ds+V(\sigma_j,Y_{\sigma_j})\right].
\]
Since $X_\sigma=Y_\sigma$, the quadratic value bound yields
\begin{align}\label{eq:RCE-bound}
 |\mathbb E\mathcal J(\sigma_j)-\mathbb E\mathcal J(\sigma)|
 \le{}&C\mathbb E[(|X_{\sigma_j}|+|Y_{\sigma_j}|)|X_{\sigma_j}-Y_{\sigma_j}|]\notag\\
 &+\mathbb E\int_\sigma^{\sigma_j}
 (|\ell(s,X_s,u_s)|+|\ell(s,Y_s,v_s)|)ds\longrightarrow0.
\end{align}
Here right continuity and the $\mathcal S^2$ estimates control the first term; absolute integrability controls the second.

For $\theta\in\mathcal T$, extend the system by
\[
 \widehat{\mathcal J}(\theta)
 =\mathbb E_\theta[\mathcal J(\theta\vee\tau)],\qquad
 |\widehat{\mathcal J}(\theta)|\le\mathbb E_\theta D_u.
\]
The tower property preserves consistency and the submartingale property; the bound gives class D. For $\theta_j\downarrow\theta$, \eqref{eq:RCE-bound} at $\theta_j\vee\tau$ gives
$\mathbb E\widehat{\mathcal J}(\theta_j)\to\mathbb E\widehat{\mathcal J}(\theta)$.
Doob regularization \cite{HeWangYan1992,Protter2005} yields a c\`adl\`ag class-D submartingale $\mathsf J$ agreeing at deterministic times. Upper dyadic approximations satisfy
\[
 \mathsf J_{\theta_j}=\widehat{\mathcal J}(\theta_j)\longrightarrow \mathsf J_\theta
 \quad\text{in }L^1,\qquad
 \widehat{\mathcal J}(\theta)\le\mathbb E_\theta \mathsf J_{\theta_j}.
\]
Taking an almost surely convergent subsequence of the conditional expectations gives
$\widehat{\mathcal J}(\theta)\le \mathsf J_\theta$. Equality of expectations then gives equality a.s. In particular,
$\mathsf J_\theta=\mathcal J(\theta)$ for every $\theta\ge\tau$.

\textbf{Step 2. Extraction on a jump interval.}
Put $r_k=\tau_k\wedge T$ and $b_k=\tau_{k+1}\wedge T$. Let $\mathsf F^k$ be the zero-control state matrix with initial value $I_n$ at $r_k$. The matrix $\Phi^k$ solves \eqref{eq:continuous-reference}. Unlike $\mathsf F^k$, it has continuous paths and provides invertible coordinates.
Extend $\Phi^k,\Psi^k$ by $I_n$ before $r_k$ and set them equal to $I_n$ when $r_k=T$.
For $x=e_i,e_i+e_j$, Step 1 aggregates
\[
 Z^{k,x}(s)=V(s,\mathsf F^k(s)x)+
 \int_{r_k}^s(\mathsf F^k(r)x)^\top Q(r)\mathsf F^k(r)x\,dr,
 \quad r_k\le s\le T.
\]
Use the extension of Step~1 for $s<r_k$. Set $Z^k_{ii}=Z^{k,e_i}$ and
$Z^k_{ij}=\tfrac12(Z^{k,e_i+e_j}-Z^{k,e_i}-Z^{k,e_j})$ for $i\ne j$.
Each entry is a linear combination of class-D submartingales. Hence $Z^k$ is special, and at every stopping time $\theta\ge r_k$,
\begin{equation}\label{eq:polarized-reference-system}
 Z^k(\theta)=(\mathsf F^k(\theta))^\top P_{\mathrm v}(\theta)\mathsf F^k(\theta)
 +\int_{r_k}^\theta(\mathsf F^k(r))^\top Q(r)\mathsf F^k(r)dr.
\end{equation}
Define the global semimartingale
\begin{equation}\label{eq:local-chart}
 \widetilde P^k(s)=(\Psi^k(s))^\top
 \left[Z^k(s)-\int_0^s\mathbf1_{(r_k,T]}(r)
 (\mathsf F^k(r))^\top Q(r)\mathsf F^k(r)dr\right]\Psi^k(s).
\end{equation}
On $\{r_k<T\}$ the two matrices satisfy
\[
 \begin{aligned}
 \mathsf F^k(s)&=\Phi^k(s), &&r_k\le s<b_k,\\
 \mathsf F^k(b_k)&=(I_n+E(b_k,e))\Phi^k(b_k),
 &&\text{at an atom }(b_k,e).
 \end{aligned}
\]
Thus \eqref{eq:polarized-reference-system} implies
\begin{equation}\label{eq:chart-agreement}
 \widetilde P^k(\theta)=P_{\mathrm v}(\theta)
 \quad\text{on }\{\tau_k\le\theta<b_k\}.
\end{equation}
At that atom, substitution in \eqref{eq:local-chart} gives
\begin{equation}\label{eq:chart-endpoint}
 \widetilde P^k(b_k)=(I_n+E(b_k,e))^\top P_{\mathrm v}(b_k)(I_n+E(b_k,e)).
\end{equation}
If $b_k=T$ and there is no atom at $T$, then $\widetilde P^k(T)=G$; this alternative holds a.s. at the deterministic terminal time.

\textbf{Step 3. Construction of the global value process.}
Define
\begin{align}
 I^k_t&=\widetilde P^k(t\wedge b_k)-\widetilde P^k(t\wedge r_k)
       =\int_0^t\mathbf1_{(r_k,b_k]}(s)d\widetilde P^k_s,\label{eq:chart-increment}\\
 c_k&=\mathbf1_{\{\tau_k<T\}}
       [P_{\mathrm v}(r_k)-\widetilde P^{k-1}(r_k)],\qquad k\ge1,\notag\\
 \mathsf P(t)&=P_{\mathrm v}(0)+\sum_{k\ge0}I^k_t+
                         \sum_{k\ge1}c_k\mathbf1_{\{r_k\le t\}}.\label{eq:global-gluing}
\end{align}
Stopped-process calculus gives the semimartingale property of $I^k$.
At $r_k<T$, the endpoint value $\widetilde P^{k-1}(r_k)$ is replaced by $P_{\mathrm v}(r_k)$.
Using \eqref{eq:chart-agreement} on each interval and $P_{\mathrm v}(T)=G$ therefore gives
\begin{equation}\label{eq:P-aggregation-identity}
 \mathsf P(\theta)=P_{\mathrm v}(\theta)\quad\text{a.s. for every }\theta\in\mathcal T.
\end{equation}
The essential bound on $E$ holds at all actual atoms, by compensation. Consequently, \eqref{eq:chart-endpoint} gives
\[
 |c_k|\le C_1,\qquad
 \mathbb E\sum_{k:\tau_k<T}|c_k|\le C_1T\nu(\mathbb Z)<\infty.
\]
The correction sum is adapted and of integrable variation. For every $m\ge1$, the stopped process is the finite semimartingale sum
\[
 \mathsf P^{r_m}_t=P_{\mathrm v}(0)
 +\sum_{k=0}^{m-1}(I^k)^{r_m}_t
 +\sum_{k=1}^{m}c_k\mathbf1_{\{r_k\le t\wedge r_m\}}.
\]
Finite activity implies $r_m\uparrow T$. On each path, $r_m=T$ for all sufficiently large $m$. The local characterization of semimartingales therefore makes $\mathsf P$ a c\`adl\`ag semimartingale. Equation~\eqref{eq:P-aggregation-identity} at rational times and right continuity yield $\sup_{t\le T}|\mathsf P(t)|\le C$ a.s. In particular $|\Delta\mathsf P|\le2C$. A c\`adl\`ag path has finitely many jumps larger than one on $[0,T]$. Their total size is bounded by $2C$ times their count. Stopping this count makes the total size integrable. The large-jump criterion \cite{Protter2005} makes $\mathsf P$ special, with canonical decomposition
\[
 \mathsf P=\mathsf P(0)+\mathcal B+M,\qquad \mathcal B_0=M_0=0,
\]
where $\mathcal B$ is predictable and locally integrable in variation, and $M$ is a local martingale.

\textbf{Step 4. Martingale representation.}
Put $c=\|\mathsf P\|_\infty$. Fix a predictable stopping time $\sigma>0$. Write $\mathcal F_{\sigma-}$ for the pre-$\sigma$ sigma-algebra and $\llbracket\sigma\rrbracket=\{(s,\omega):s=\sigma(\omega)\le T\}$ for its graph. Set
$A_j=\{\sigma\le T,|\Delta\mathcal B_\sigma|\le j\}\in\mathcal F_{\sigma-}$.
The process
$\int_0^t\mathbf1_{A_j}\mathbf1_{\llbracket\sigma\rrbracket}(s)\,dM_s
 =\mathbf1_{A_j}\Delta M_\sigma\mathbf1_{\{\sigma\le t\}}$
is a local martingale bounded by $2c+j$. Hence
\[
 \mathbb E[\mathbf1_{A_j}\Delta M_\sigma\mid\mathcal F_{\sigma-}]=0,
 \qquad
 \mathbf1_{A_j}\Delta\mathcal B_\sigma
 =\mathbf1_{A_j}\mathbb E[\Delta\mathsf P_\sigma\mid\mathcal F_{\sigma-}].
\]
Letting $j\uparrow\infty$ gives
\begin{equation}\label{eq:predictable-jump-bound}
 |\Delta\mathcal B_\sigma|\le2c.
\end{equation}
The jumps of $\mathcal B$ lie on countably many predictable stopping-time graphs, so $|\Delta M|\le4c$ outside one null set. For
$\eta_j=T\wedge\inf\{t:|M_t|>j\}$,
\[
 \sup_{t\le T}|M_{t\wedge\eta_j}|\le j+4c,
 \qquad
 M_{t\wedge\eta_j}=\int_0^t\Lambda^j_s\,dW_s+
                  \int_0^t\!\int_{\mathbb Z}\Xi^j(s,e)\widetilde\mu(ds,de).
\]
The square-integrable representation property \cite{HeWangYan1992} and its uniqueness identify the integrands on overlaps. Since $\eta_j=T$ eventually on each c\`adl\`ag path, these consistent integrands define a predictable pair $(\Lambda,\Xi)$ with pathwise finite square integrals on $[0,T]$.

Poisson times are totally inaccessible, while $\mathcal B$ is predictable. At an actual atom $(s,e)$, therefore,
$\Xi(s,e)=\Delta M_s=\Delta\mathsf P_s$ and $|\Xi(s,e)|\le2c$.
For the nonnegative predictable field $f=(|\Xi|-2c)^+\wedge1$, compensation gives
\[
 0=\mathbb E\int_0^T\!\int_{\mathbb Z}f\,\mu(ds,de)
  =\mathbb E\int_0^T\!\int_{\mathbb Z}f\,\nu(de)ds.
\]
This proves \eqref{eq:Xi-bounded}. Symmetry follows from uniqueness of the canonical decomposition and representation. Renaming $\mathsf P$ as $P$ proves the theorem.
\label{end:thm:semimartingale_P}
\end{proof}

\Needspace{7\baselineskip}
\section{Solvability of the Stochastic Riccati Equation with Jumps}\label{sec:SREJ}
Theorem~\ref{thm:semimartingale_P} supplies a value semimartingale with a general predictable finite-variation part. We now identify that part from the quadratic drift of $\mathcal J$ in \eqref{eq:cost-to-go}.

\subsection{The Riccati coefficients}\label{subsec4.1}
For symmetric $P(s),\Lambda(s),\Xi(s,e)\in\mathbb S^n$, use the coefficients $\mathscr N,\mathscr M$ from \eqref{eq:introN} and define $\mathscr H$ by
\begin{align}\label{eq:HMN}
\mathscr H={}&A^\top P_-+P_-A+\Lambda C+C^\top\Lambda+C^\top P_-C+Q\notag\\
 &+\int_{\mathbb Z}\bigl[\Xi E+E^\top\Xi+E^\top P_-E+E^\top\Xi E\bigr]\nu(de).
\end{align}

Here $\mathscr H(s)\in\mathbb S^n$; all terms are evaluated at time $s$, and the integrands at $(s,e)$. These are the state, cross and control blocks in the quadratic drift:
\[
 x^\top\mathscr Hx+2x^\top\mathscr Mu+u^\top\mathscr Nu
 =\begin{pmatrix}x\\u\end{pmatrix}^{\!\top}
   \begin{pmatrix}\mathscr H&\mathscr M\\\mathscr M^\top&\mathscr N\end{pmatrix}
   \begin{pmatrix}x\\u\end{pmatrix}.
\]
The stochastic Riccati equation with jumps (SREJ) is

{
\begin{equation}\label{eq:SREJ}
 \left\{\begin{aligned}
 dP(s)&=-\bigl[\mathscr H(s)-\mathscr M(s)\mathscr N(s)^{-1}
                        \mathscr M(s)^\top\bigr]ds
       +\Lambda(s)dW(s)+\int_{\mathbb Z}\Xi(s,e)\widetilde\mu(ds,de),\\
 P(T)&=G.
 \end{aligned}\right.
\end{equation}
}

\paragraph{Solution class}
A \emph{bounded solution} of \eqref{eq:SREJ} is a triple $(P,\Lambda,\Xi)$
with $P$ adapted, symmetric and c\`adl\`ag, $\|P\|_\infty<\infty$,
and symmetric predictable integrands satisfying
\[
 \int_0^T\left(|\Lambda_s|^2+
       \int_{\mathbb Z}|\Xi(s,e)|^2\nu(de)\right)ds<\infty
 \quad\text{a.s.}
\]
The matrix $\mathscr N$ must be invertible $ds\otimes d\mathbb P$-a.e.
Set $\Psi_P:=-(\mathscr H-\mathscr M\mathscr N^{-1}\mathscr M^\top)$.
We require $\int_0^T|\Psi_P|ds<\infty$ a.s.
The integrated equation must hold up to indistinguishability, with $P_T=G$ a.s.
The stochastic integrals are interpreted by localization on $[0,T]$.
We identify $P$ up to indistinguishability and the two integrands modulo
$ds\otimes d\mathbb P$ and $ds\otimes\nu(de)\otimes d\mathbb P$, respectively.

{A solution is \emph{strongly regular} if
$\mathscr N\succeq\eta I_m$ $ds\otimes d\mathbb P$-a.e. for some deterministic
$\eta>0$, which may depend on the solution.
The solution constructed from the value kernel, called the value solution, satisfies this inequality with the constant $\delta$ from \UC{}.}

\subsection{Positivity and identification of the entire finite-variation part}
The proof first evaluates the quadratic drift along states started at Poisson times and uses auxiliary-cost comparison to prove uniform positive definiteness of $\mathscr N$. Optimality then identifies the feedback and forces the remaining nonnegative matrix measure to vanish, using state matrices that are invertible on each interval between jumps. For a symmetric matrix-valued finite-variation measure $d\mathcal C$, the notation $d\mathcal C\succeq0$ means that $y^\top d\mathcal C\,y$ is a nonnegative scalar measure for every $y\in\mathbb R^n$.
\begin{lem}\label{lem:residual-measure}
	Under the assumptions of Theorem~\ref{thm:semimartingale_P}, the processes constructed there satisfy
	\begin{equation}\label{eq:N-positive}
		\mathscr N(s)\succeq\delta I_m,\qquad ds\otimes d\mathbb P\text{-a.e.}
	\end{equation}
	Moreover, the predictable finite-variation term in \eqref{eq:P-general-FV} is absolutely continuous and
	\begin{equation}\label{eq:B-density}
		d\mathcal B(s)=-\bigl[\mathscr H(s)-\mathscr M(s)\mathscr N(s)^{-1}
		\mathscr M(s)^\top\bigr]ds.
	\end{equation}
	Thus $(P,\Lambda,\Xi)$ solves \eqref{eq:SREJ}. For every open-loop optimal pair starting at a stopping time,
	\begin{equation}\label{eq:optimizer-stationarity}
		u^{\tau,\xi}(s)=-\mathscr Q(s)X^{\tau,\xi}(s-),\qquad
		\mathscr Q:=\mathscr N^{-1}\mathscr M^\top,
		\quad ds\otimes d\mathbb P\text{-a.e. on }(\tau,T].
	\end{equation}
\end{lem}
\begin{proof}
	\textbf{Step 1. The predictable drift measure.}
	Fix $(\tau,\xi;u)$, write $X=X^{\tau,\xi;u}$, and set
	$\mathcal J^{\tau,\xi;u}_t=X_t^\top P_tX_t+\int_\tau^t\ell(s,X_s,u_s)ds$.
	With $x=X_-$, $p=P_-$, $g=Cx+Du$ and $h(e)=E(e)x+F(e)u$, It\^o's formula for $X^\top PX$ gives
	\begin{equation}\label{eq:measure-drift}
		d\mathcal J^{\tau,\xi;u}
		=dM^u+d\mathfrak D^u,\qquad
		d\mathfrak D^u
		=x^\top d\mathcal B\,x+
		\bigl[x^\top\mathscr Hx+2x^\top\mathscr Mu
		+u^\top\mathscr Nu\bigr]ds,
	\end{equation}
	where $M^u$ is a local martingale and $\mathfrak D^u_\tau=0$ is predictable and of finite
	variation. The jump correction in its $ds$-density is
	\begin{equation}\label{eq:jump-algebra}
		\begin{aligned}
			&(x+h)^\top(p+\Xi)(x+h)-x^\top px
			-2x^\top ph-x^\top\Xi x\\
			&\hspace{34mm}=2x^\top\Xi h+h^\top(p+\Xi)h.
		\end{aligned}
	\end{equation}
	Together with the Brownian covariation, this yields exactly
	\eqref{eq:introN} and \eqref{eq:HMN}. A predictable jump of $\mathcal B$ cannot coincide
	with a Poisson jump of $X$. The local martingale integrands in
	\eqref{eq:measure-drift} are
	\[
	\mathfrak b^u=x^\top\Lambda x+2x^\top pg,
	\qquad
	\mathfrak j^u(e)=(x+h(e))^\top(p+\Xi(e))(x+h(e))-x^\top px.
	\]
	By \eqref{eq:Xi-bounded} and the state estimate,
	\[
	\begin{aligned}
		&\int_\tau^T|\mathfrak b^u|^2ds<\infty\quad\text{a.s.},\\
		&\mathbb E\int_\tau^T\!\int_{\mathbb Z}|\mathfrak j^u(e)|\,\nu(de)ds
		\le C\mathbb E\left[\sup_{\tau\le s\le T}|X_s|^2
		+\int_\tau^T|u_s|^2ds\right]<\infty.
	\end{aligned}
	\]
	Thus the Brownian term is a local martingale and the compensated
	Poisson term is an $L^1$ martingale.
	By \eqref{eq:P-aggregation-identity}, the process $\mathcal J^{\tau,\xi;u}$ aggregates the
	class-D system in Theorem~\ref{thm:DPP}. Doob--Meyer and uniqueness
	of the canonical decomposition identify its predictable part with
	$\mathfrak D^u$. For an optimal control this part vanishes, so
	\[
	d\mathfrak D^u\ge0\quad\text{a.s. for each fixed }u,
	\qquad d\mathfrak D^{u^*}=0\quad\text{a.s.}
	\]
	These identities hold as random measures on $(\tau,T]$.
	
	\textbf{Step 2. Positivity of the drift measure.}\label{loc:pointwise}
	Let $|\mathcal B|_t=\operatorname{Var}_{[0,t]}(\mathcal B)$ and set
	\[
	 \mathfrak a_t=|\mathcal B|_t+
	  \int_0^t(|\mathscr H_s|+|\mathscr M_s|+|\mathscr N_s|)ds,
	 \qquad
	 \chi_j=T\wedge\inf\{t:\mathfrak a_t\ge j\}.
	\]
	Theorem~\ref{thm:semimartingale_P} gives $\mathfrak a_T<\infty$ a.s.; by \eqref{eq:predictable-jump-bound},
	\[
	 \mathfrak a_{\chi_j}\le j+2\|P\|_\infty,
	 \qquad \chi_j=T\quad\text{eventually on each path}.
	\]
	The predictable Radon--Nikodym theorem therefore applies on each stopped interval and gives
	\[
	da=ds+d|\mathcal B|,\qquad
	d\mathcal B=\beta\,da,\qquad ds=\rho\,da,\qquad \rho\ge0.
	\]
	Here $a_t=t+|\mathcal B|_t$. The notation $da$-a.e. refers to its pathwise Lebesgue--Stieltjes measure, for $\mathbb P$-almost every $\omega$. The densities agree $da$-almost everywhere on overlapping stopped intervals and therefore define common densities on their union.
	For $k\ge0$ with $\tau_k<T$ and $(z,v)\in\mathbb Q^n\times\mathbb Q^m$ (rational coordinates),
	consider the state with initial pair $(\tau_k\wedge T,z)$ under the constant control $v$.
	The inequality $d\mathfrak D^u\ge0$ yields, on $(\tau_k,b_k]$,
	\begin{equation}\label{eq:restart-polynomial-test}
 \begin{gathered}
 q_s(X^{k,z;v}_{s-},v)\ge0\quad da\text{-a.e. on }(\tau_k,b_k],\\
 q_s(y,v)=y^\top(\beta+\rho\mathscr H)y
                     +2\rho y^\top\mathscr Mv+\rho v^\top\mathscr Nv .
 \end{gathered}
\end{equation}
	Fix $j,k$ and rational $(z,v)$. Outside a $\mathbb P$-null set, \eqref{eq:restart-polynomial-test} holds $da(\omega)$-almost everywhere on $(\tau_k,b_k]\cap(0,\chi_j]$. The choices of $j,k,z,v$ are countable, so one $\mathbb P$-null set suffices for all of them. By
	Lemma~\ref{lem:restarted-reference}, including its left-limit identity at
	$b_k$,
	\[
	X^{k,z;v}_{s-}=\Phi^k(s)z+X^{k,0;v}_{s-},\qquad
	\det\Phi^k(s)\ne0,\quad s\in(\tau_k,b_k].
	\]
	Density of $\mathbb Q^n$ and continuity of $q_s$ first in $y$ and then
	in $v$ give
	\[
	q_s(y,v)\ge0\quad\text{for all }(y,v)\in\mathbb R^n\times\mathbb R^m,
	\quad da\text{-a.e. on }(\tau_k,b_k].
	\]
	The intervals $(\tau_k,b_k]$ cover $(0,T]$ pathwise, and $\chi_j=T$ eventually on each path. Removing the localization gives the inequality $da(\omega)$-a.e. on $(0,T]$, for almost every path. In particular, this is a random-measure assertion; its $ds\otimes d\mathbb P$ consequence follows from $ds\le da$. Put
	\[
	b=\frac{\beta}{\rho}\ \text{on }\{\rho>0\},\qquad
	d\mathcal B^s=\mathbf1_{\{\rho=0\}}\beta\,da,
	\qquad d\mathcal B=b\,ds+d\mathcal B^s.
	\]
	On $\{\rho=0\}$ take $v=0$; on $\{\rho>0\}$ divide by $\rho$.
	The resulting inequalities are
	\begin{equation}\label{eq:polynomial-positive}
 \begin{gathered}
 d\mathcal B^s\succeq0\quad\text{as a matrix measure},\\
 \begin{pmatrix}b+\mathscr H&\mathscr M\\
                 \mathscr M^\top&\mathscr N\end{pmatrix}\succeq0
 \quad ds\otimes d\mathbb P\text{-a.e.},\\
 \mathscr N\succeq0\quad ds\otimes d\mathbb P\text{-a.e.}
 \end{gathered}
\end{equation}
	
	\textbf{Step 3. Uniform positivity and the post-jump comparison.}\label{rev:step3}
	Fix $0<\varepsilon<\delta$ and put $\widehat R=R-\varepsilon I_m$, using hats for the auxiliary cost, value and coefficients. Then
	\[
	\widehat J(t,0;u)=J(t,0;u)-\varepsilon\mathbb E_t
	\int_t^T|u_s|^2ds
	\ge(\delta-\varepsilon)\mathbb E_t\int_t^T|u_s|^2ds.
	\]
	Proposition~\ref{prop:quadratic}, Theorem~\ref{thm:semimartingale_P},
	and Steps~1--2 applied to this auxiliary problem give
	$(\widehat P,\widehat\Lambda,\widehat\Xi)$ with
	$\widehat{\mathscr N}\succeq0$.
	For each deterministic $t$ and $x$, $V(t,x)\ge\widehat V(t,x)$.
	Using rational $t<T$ and $x$, the c\`adl\`ag versions and $P_T=\widehat P_T=G$ give, outside one
	null set,
	\[
	P_t-\widehat P_t\succeq0\ (0\le t\le T),\qquad
	P_{s-}-\widehat P_{s-}\succeq0\ (0<s\le T).
	\]
	With $\lambda_{\min}$ denoting the smallest eigenvalue of a symmetric matrix, set
	\[
	K(s,e)=P_{s-}-\widehat P_{s-}+\Xi(s,e)-\widehat\Xi(s,e),
	\qquad f(s,e)=\min\{1,[\lambda_{\min}(K(s,e))]^-\}.
	\]
	At every Poisson atom $(s,e)$, the predictable finite-variation
	parts have no jump; hence
	\[
	K(s,e)=P_s-\widehat P_s\succeq0,
	\qquad f(s,e)=0\quad\mu\text{-a.e.}
	\]
	Predictability and compensation now give
	\begin{equation}\label{eq:compensator-comparison}
		\begin{gathered}
			0=\mathbb E\int_0^T\!\int_{\mathbb Z}f\,\mu(ds,de)
			=\mathbb E\int_0^T\!\int_{\mathbb Z}f\,\nu(de)ds,\\
			K\succeq0\quad ds\otimes\nu(de)\otimes d\mathbb P\text{-a.e.}
		\end{gathered}
	\end{equation}
	The integral of $F^\top KF$ is well defined by the boundedness of
	$P,\widehat P,\Xi,\widehat\Xi,F$ and finite $\nu$. Thus
	\begin{align}\label{eq:shifted-N}
		\mathscr N-\widehat{\mathscr N}
		&=\varepsilon I_m+D^\top(P_--\widehat P_-)D+
		\int_{\mathbb Z}F^\top K F\,\nu(de)\notag\\*
		&\succeq\varepsilon I_m,
		\qquad \mathscr N\succeq\varepsilon I_m.
	\end{align}
	Taking a countable sequence $\varepsilon\uparrow\delta$ proves
	\eqref{eq:N-positive}. The boundedness of $P,\Xi,D,F,R$ also yields
	$\|\mathscr N\|_{\rm op}\le C$.
	
	\Needspace{8\baselineskip}
	\textbf{Step 4. Identification of the drift.}\label{rev:step4}
	Define $\mathscr Q=\mathscr N^{-1}\mathscr M^\top$. The residual between the predictable finite-variation part of $P$ and the Riccati drift measure is
	\begin{equation}\label{eq:residual-measure}
		d\mathcal K=d\mathcal B+
		[\mathscr H-\mathscr M\mathscr N^{-1}\mathscr M^\top]ds,
		\qquad \mathcal K(0)=0.
	\end{equation}
	By \eqref{eq:N-positive}, $\mathscr N$ is positive definite. Taking the Schur complement of this block in \eqref{eq:polynomial-positive} gives
	\[
	d\mathcal K=d\mathcal B^s+\Theta\,ds\succeq0,
	\qquad
	\Theta=b+\mathscr H-\mathscr M\mathscr N^{-1}\mathscr M^\top
	\succeq0.
	\]
	Set $m_0=\mathscr M-\Lambda D$, which is bounded by \eqref{eq:introN}. Then
	\begin{equation}\label{eq:gain-local-square}
		|\mathscr Q|^2\le C(1+|\Lambda D|^2),\qquad
		\int_0^T\bigl(|\mathscr Q|^2+|\Theta|\bigr)ds<\infty
		\quad\text{a.s.}
	\end{equation}
	For any existing optimal pair $(X,u)$, the martingale assertion of
	Theorem~\ref{thm:DPP} and completion of squares give
	\begin{equation}\label{eq:optimal-measure-zero}
		0=d\mathfrak D^u
		=X_-^\top d\mathcal K X_-
		+(u+\mathscr QX_-)^\top\mathscr N(u+\mathscr QX_-)ds.
	\end{equation}
	Both terms are nonnegative and $\mathscr N\succeq\delta I_m$;
	therefore
	\[
	X_-^\top d\mathcal K X_-=0,
	\qquad \int_\tau^T|u+\mathscr QX_-|^2ds=0\quad\text{a.s.}
	\]
	This proves \eqref{eq:optimizer-stationarity}. Predictability,
	Brownian isometry, and Poisson compensation show that substitution of
	$u=-\mathscr QX_-$ preserves both stochastic integrals, after
	localization. For this pair, $\mathscr QX_-$ belongs to the admissible
	$L^2$ control class because it agrees with $-u$ almost everywhere.
	
	Set $r_k=\tau_k\wedge T$, let $\gamma^{k,i}=X^{r_k,e_i}$ be the optimal state from the $i$th coordinate vector, and put
	$\Gamma^k=(\gamma^{k,1},\ldots,\gamma^{k,n})$.
	On $\{\tau_k<T\}$, $\Gamma^k_{\tau_k}=I_n$, and \eqref{eq:optimal-measure-zero}
	applied to each column yields
	\begin{equation}\label{eq:Gamma-measure-zero}
 (\gamma^{k,i}_-)^\top d\mathcal K\gamma^{k,i}_-=0\ (1\le i\le n)
 \quad\Longrightarrow\quad
 (\Gamma^k_-)^\top d\mathcal K\Gamma^k_-=0
 \quad\text{on }(\tau_k,b_k].
\end{equation}
	Indeed, every positive semidefinite matrix with zero diagonal vanishes; apply this to the density of the matrix measure with respect to its trace.
	Between $\tau_k$ and $b_k$, $\Gamma^k$ satisfies
	\begin{equation}\label{eq:between-jump-flow}
 \begin{gathered}
 d\Gamma^k=\overline A\Gamma^k ds+\overline C\Gamma^k dW,\quad
 \Gamma^k_{\tau_k}=I_n,\\
 \overline A=A-B\mathscr Q-\int_{\mathbb Z}(E-F\mathscr Q)\nu(de),
 \qquad \overline C=C-D\mathscr Q .
 \end{gathered}
\end{equation}
	By \eqref{eq:gain-local-square},
	\[
	\int_0^T(|\overline A|+|\overline C|^2)ds<\infty\quad\text{a.s.}
	\]
	Continue the Brownian-driven linear equation \eqref{eq:between-jump-flow} to
	the right endpoint $b_k$ and denote its solution by
	$\overline\Gamma^k$. Its inverse candidate $\Upsilon^k$ solves
	\[
	d\Upsilon^k=\Upsilon^k(-\overline A+\overline C\,\overline C)ds
	-\Upsilon^k\overline C\,dW,\qquad \Upsilon^k_{\tau_k}=I_n.
	\]
	Applying It\^o's formula to $\Upsilon^k\overline\Gamma^k$ after localization gives
	\[
	\Upsilon^k_s\overline\Gamma^k_s=I_n\quad(\tau_k\le s\le b_k),
	\qquad \Gamma^k_{s-}=\overline\Gamma^k_s
	\quad(s\in(\tau_k,b_k]).
	\]
	
	Finally, let $d\alpha=\operatorname{tr}(d\mathcal K)$ and choose
	$\varkappa=d\mathcal K/d\alpha$, so that $\varkappa\succeq0$ and
	$\operatorname{tr}\varkappa=1$ $d\alpha$-a.e. Equations
	\eqref{eq:Gamma-measure-zero} and the inverse identity imply, on each
	$(\tau_k,b_k]$,
	\[
	0=(\Gamma^k_-)^\top\varkappa\Gamma^k_-
	\quad\Longrightarrow\quad
	\varkappa=(\Upsilon^k)^\top(\Gamma^k_-)^\top\varkappa\Gamma^k_-\Upsilon^k=0
	\quad d\alpha\text{-a.e.}
	\]
	Since $\operatorname{tr}\varkappa=1$ $d\alpha$-a.e.,
	\[
	 \alpha((\tau_k,b_k])=\int_{(\tau_k,b_k]}\operatorname{tr}\varkappa\,d\alpha=0.
	\]
	Taking a common null set over $k$ and using
	$\bigcup_{k:\tau_k<T}(\tau_k,b_k]=(0,T]$, we obtain
	\[
	d\alpha=0,\qquad d\mathcal K=0,\qquad
	d\mathcal B=-[\mathscr H-\mathscr M\mathscr N^{-1}\mathscr M^\top]ds.
	\]
	Combining this identity with \eqref{eq:P-general-FV} and $P(T)=G$
	proves \eqref{eq:SREJ}.
	\label{end:lem:residual-measure}
\end{proof}

\Needspace{4\baselineskip}
\subsection{Existence and global regularity}
Lemma~\ref{lem:residual-measure} has identified the Riccati equation and its feedback coefficient. Conditional energy estimates also appear in the nonnegative-weight jump theory \cite{ZhangDongMeng2020}. We establish them for the present indefinite value kernel, without a nonsingularity assumption on the jump multipliers. They give global integrability of the martingale integrands and Riccati drift. The higher-moment bounds follow from the conditional energy estimate by iteration.
\begin{thm}\label{thm:SREJ-existence-uniqueness}
Let Assumptions~\ref{ass:AS1}--\ref{ass:AS2} and \UC{} hold. Let $(P,\Lambda,\Xi)$ be the value-kernel triple constructed in Theorem~\ref{thm:semimartingale_P}. This triple solves \eqref{eq:SREJ} and satisfies \eqref{eq:N-positive}. There is a deterministic $C$ such that
\begin{equation}\label{eq:BMO-bound}
 \mathbb E_\tau\int_\tau^T
 \left(|\Lambda(s)|^2+\int_{\mathbb Z}|\Xi(s,e)|^2\nu(de)\right)ds\le C
 \qquad\text{a.s., for every }\tau\in\mathcal T.
\end{equation}
Set $\Psi_P=-(\mathscr H-\mathscr M\mathscr N^{-1}\mathscr M^\top)$. For each $p\ge1$, there is a deterministic $C_p$ such that the following bound holds for every $\tau\in\mathcal T$:
\begin{equation}\label{eq:global-moments}
 \mathbb E_\tau\!\left[
 \left(\int_\tau^T|\Lambda|^2ds\right)^p+
 \left(\int_\tau^T\!\int_{\mathbb Z}|\Xi|^2\nu(de)ds\right)^p+
 \left(\int_\tau^T|\Psi_P|ds\right)^p\right]\le C_p
 \quad\text{a.s.}
\end{equation}
\end{thm}
\begin{proof}
Lemma~\ref{lem:residual-measure} gives the equation and strong regularity. Choose a deterministic $c$ so that $U=P+cI_n\succeq I_n$, and write
\[
 \mathscr L=\mathscr M\mathscr N^{-1}\mathscr M^\top\succeq0,
 \qquad h_0:=\mathscr H-\Lambda C-C^\top\Lambda,\qquad |U|+|h_0|\le C_1.
\]
The shift of $P$ gives $\operatorname{tr}(U_-\mathscr L)\ge\operatorname{tr}\mathscr L\ge0$ in the energy estimate. Moreover, $|\mathscr L|\le C(1+|\Lambda|^2)$, so the stopping times below also make the drift integrable. For $\tau\in\mathcal T$ set
$\rho_j=T\wedge\inf\{t\ge\tau:\int_\tau^t|\Lambda_s|^2ds\ge j\}$.
It\^o's formula for $|U|^2$ gives the martingale terms
\[
 2\operatorname{tr}(U_-\Lambda)dW+
 \int_{\mathbb Z}[2\operatorname{tr}(U_-\Xi)+|\Xi|^2]\widetilde\mu(ds,de).
\]
Their stopped increments have zero conditional expectation: the stopped Brownian integral has quadratic variation at most $Cj$, and the jump integrand is bounded by \eqref{eq:Xi-bounded}. Therefore
\begin{align}\label{eq:energy-ito}
 &\mathbb E_\tau\int_\tau^{\rho_j}
 \left(|\Lambda|^2+\int_{\mathbb Z}|\Xi(e)|^2\nu(de)
       +2\operatorname{tr}(U_-\mathscr L)\right)ds\notag\\
 &\hspace{12mm}=\mathbb E_\tau|U_{\rho_j}|^2-|U_\tau|^2
       +2\mathbb E_\tau\int_\tau^{\rho_j}\operatorname{tr}(U_-\mathscr H)ds.
\end{align}
Using
$\operatorname{tr}(U_-\mathscr L)\ge0$ and
$2\operatorname{tr}(U_-\mathscr H)\le C_2+\frac12|\Lambda|^2$, we obtain
\[
 \tfrac12\mathbb E_\tau\int_\tau^{\rho_j}|\Lambda|^2ds
 +\mathbb E_\tau\int_\tau^{\rho_j}\!\int_{\mathbb Z}|\Xi(e)|^2\nu(de)ds\le C_3.
\]
Conditional monotone convergence proves \eqref{eq:BMO-bound}.

For $a_s^{(N)}=|\Lambda_s|^2\wedge N$ and
$A_t^{(N)}=\int_t^T a_s^{(N)}ds$, put $C_4=2C_3$. For each integer $k\ge1$,
\[
 (A_\tau^{(N)})^k
 =k\int_\tau^T a_s^{(N)}(A_s^{(N)})^{k-1}ds.
\]
Conditional Fubini, the tower property, and induction using the conditional energy bound give
\[
 \begin{aligned}
 \mathbb E_\tau\big[(A_\tau^{(N)})^k\big]
 &=k\mathbb E_\tau\int_\tau^T a_s^{(N)}
                   \mathbb E_s\big[(A_s^{(N)})^{k-1}\big]ds\\
 &\le k!C_4^{\,k},\qquad k\ge1.
 \end{aligned}
\]
Let $N\uparrow\infty$. For arbitrary $p\ge1$, choose an integer $k\ge p$ and apply conditional Jensen's inequality to obtain
\[
 \mathbb E_\tau\left(\int_\tau^T|\Lambda_s|^2ds\right)^p
 \le(k!C_4^k)^{p/k}.
\]
The remaining conditional estimates follow from
\[
 \int_0^T\!\int_{\mathbb Z}|\Xi|^2\nu(de)ds
 \le4\|P\|_\infty^2T\nu(\mathbb Z),\qquad
 |\Psi_P|\le C(1+|\Lambda|^2).
\]
This proves \eqref{eq:global-moments}.
\label{end:thm:SREJ}
\end{proof}

\section{Feedback Synthesis and Optimal Control}\label{sec:feedback}

We turn the optimality identity of Lemma~\ref{lem:residual-measure} into a closed-loop synthesis and a verification theorem.
For the value solution, put $\mathscr Q=\mathscr N^{-1}\mathscr M^\top\in\mathbb R^{m\times n}$. The closed-loop equation, started from the post-jump state $\xi$ at $\tau$, is
\begin{equation}\label{eq:closed-loop}
 \left\{\begin{aligned}
 dX^*&=(A-B\mathscr Q)X^*_-\,ds+(C-D\mathscr Q)X^*_-\,dW\\
 &\quad+\int_{\mathbb Z}(E-F\mathscr Q)X^*_-\,\widetilde\mu(ds,de),\\
 X^*(\tau)&=\xi.
 \end{aligned}\right.
\end{equation}
The left limits in the feedback are necessary for predictability.

\begin{lem}\label{lem:jump_sde_general}
Let Assumption~\ref{ass:AS1} hold. Let $K:(0,T]\times\Omega\to\mathbb R^{m\times n}$ be predictable and satisfy $\int_0^T|K|^2ds<\infty$ a.s. Fix $\tau\in\mathcal T$ and a finite $\mathcal F_\tau$-measurable initial vector $\xi:\Omega\to\mathbb R^n$. Replace $\mathscr Q$ by $K$ in \eqref{eq:closed-loop}. The resulting equation has a pathwise unique c\`adl\`ag solution on $[\tau,T]$, with stochastic integrals interpreted locally.
\end{lem}
\begin{proof}
On each jump interval the equation has coefficients
\[
 \begin{gathered}
 A_K=A-BK-\int_{\mathbb Z}(E-FK)\nu(de),\qquad C_K=C-DK,\\
 \int_0^T(|A_K|+|C_K|^2)ds<\infty\quad\text{a.s.}
 \end{gathered}
\]
The linear equation $dX=A_KX\,ds+C_KX\,dW$ has a pathwise unique continuous solution by localization of these integrals \cite{Galchuk1978,Protter2005}. At an atom $(s,e)$, set
$X_s=(I_n+E(s,e)-F(s,e)K_s)X_{s-}$.
This gives a unique forward update. For unbounded initial data, set $\xi_j=\xi\mathbf1_{\{|\xi|\le j\}}$. The corresponding solutions agree on the events $\{|\xi|\le j\}\in\mathcal F_\tau$. They therefore define the solution for $\xi$. Concatenating the finitely many intervals proves the assertion. Predictability and compensation make the updates independent of the chosen versions of the coefficients.
\label{end:lem:jump_sde_general}
\end{proof}

Let $(\overline P,\overline\Lambda,\overline\Xi)$ be a bounded strongly regular solution. Use bars for the coefficients obtained by substitution in \eqref{eq:introN}, and set $\overline{\mathscr Q}=\overline{\mathscr N}^{-1}\overline{\mathscr M}^{\top}$. Compensation at the jumps of $\overline P$ and strong regularity give
\[
 |\overline\Xi|\le2\|\overline P\|_\infty,\qquad
 |\overline{\mathscr Q}|^2\le C(1+|\overline\Lambda|^2),\qquad
 \int_0^T|\overline{\mathscr Q}|^2ds<\infty\quad\text{a.s.}
\]
Lemma~\ref{lem:jump_sde_general} therefore defines its closed-loop solution $\overline X$ on $[\tau,T]$, with stochastic integrals interpreted locally. The following condition specifies the integrability needed to use this feedback in the control problem.

\begin{defn}\label{def:admissible-class}
Let $(\overline P,\overline\Lambda,\overline\Xi)$ be a bounded strongly regular solution of \eqref{eq:SREJ}. It is \emph{feedback-admissible} if its closed-loop solution satisfies the following conditions for every $\tau\in\mathcal T$ and $\xi\in L^2(\mathcal F_\tau;\mathbb R^n)$:
\begin{equation}\label{eq:admissible-class}
 \overline X\in\mathcal S^2[\tau,T],\qquad
 -\overline{\mathscr Q}\,\overline X_-\in\mathcal U[\tau,T].
\end{equation}
\end{defn}

For the value solution, the open-loop energy estimate supplies feedback admissibility. The cost identity below then yields maximality and uniqueness in the feedback-admissible class.
\begin{thm}\label{thm:verification}
Under the assumptions of Theorem~\ref{thm:SREJ-existence-uniqueness}, the value solution is feedback-admissible. Fix $\tau\in\mathcal T$ and $\xi\in L^2(\mathcal F_\tau;\mathbb R^n)$. The closed-loop equation \eqref{eq:closed-loop} has a pathwise unique solution $X^*$. The unique open-loop optimal control and the value are
\begin{equation}\label{eq:feedback-optimum}
 u^*(s)=-\mathscr Q(s)X^*(s-),\qquad
 V(\tau,\xi)=\xi^\top P(\tau)\xi.
\end{equation}
The optimal pair satisfies \eqref{eq:optimal-estimate}. Every admissible $u$ obeys the identity
\begin{equation}\label{eq:full-completion}
 J(\tau,\xi;u)=\xi^\top P(\tau)\xi+
 \mathbb E_\tau\int_\tau^T
 (u+\mathscr QX_-)^\top\mathscr N(u+\mathscr QX_-)ds.
\end{equation}
Moreover:
\begin{enumerate}[label=(\roman*),leftmargin=2em,itemsep=2pt,topsep=3pt]
\item Every bounded strongly regular solution $(\overline P,\overline\Lambda,\overline\Xi)$ satisfies $\overline P\preceq P$.
\item In the feedback-admissible class of Definition~\ref{def:admissible-class}, $\overline P=P$ up to indistinguishability, while $\overline\Lambda=\Lambda$ $ds\otimes d\mathbb P$-a.e. and $\overline\Xi=\Xi$ $ds\otimes\nu(de)\otimes d\mathbb P$-a.e.
\end{enumerate}
\end{thm}
\begin{proof}
\textbf{Step 1. The optimal feedback.}
For each $(\tau,\xi)$, Proposition~\ref{prop:quadratic} and Lemma~\ref{lem:residual-measure} give
\[
 (X^{\tau,\xi},u^{\tau,\xi})\in\mathcal S^2[\tau,T]\times\mathcal U[\tau,T],
 \qquad u^{\tau,\xi}=-\mathscr QX^{\tau,\xi}_-\quad ds\otimes d\mathbb P\text{-a.e.}
\]
Predictability, Brownian isometry and Poisson compensation allow substitution in the state equation. Hence $X^{\tau,\xi}$ solves \eqref{eq:closed-loop}. By \eqref{eq:gain-local-square} and Lemma~\ref{lem:jump_sde_general}, that equation is pathwise unique. Thus
$X^*=X^{\tau,\xi}$, $u^*=u^{\tau,\xi}$, and \eqref{eq:optimal-estimate} proves feedback admissibility.

\textbf{Step 2. Completion of squares.}
Let $(\overline P,\overline\Lambda,\overline\Xi)$ be any bounded strongly regular solution, with $\overline{\mathscr N}\succeq\eta I_m$. Its Poisson martingale integrand satisfies
$|\overline\Xi|\le2\|\overline P\|_\infty$ by compensation. Hence
$\overline{\mathscr N}$ is bounded and
$|\overline{\mathscr Q}|^2\le C(1+|\overline\Lambda|^2)$.
For $u\in\mathcal U[\tau,T]$ with state $X$, put
\[
 Y_t=X_t^\top\overline P_tX_t+\int_\tau^t\ell(s,X_s,u_s)ds,
 \qquad q_s=(u_s+\overline{\mathscr Q}_sX_{s-})^\top
              \overline{\mathscr N}_s(u_s+\overline{\mathscr Q}_sX_{s-})\ge0.
\]
It\^o's formula, expanded in Appendix~\ref{app:product}, yields
\begin{equation}\label{eq:verification-local-identity}
 Y_t=\xi^\top\overline P_\tau\xi+\int_\tau^t q_sds+M_t^u-M_\tau^u.
\end{equation}
For its Brownian and jump integrands $\mathfrak b^u,\mathfrak j^u$,
\[
 \begin{gathered}
 \int_\tau^T|\mathfrak b^u|^2ds<\infty\quad\text{a.s.},\\
 \mathbb E\int_\tau^T\!\int_{\mathbb Z}|\mathfrak j^u(s,e)|\nu(de)ds
 \le C\mathbb E\!\left[\sup_{\tau\le s\le T}|X_s|^2+\int_\tau^T|u_s|^2ds\right]<\infty.
 \end{gathered}
\]
Set $\rho_j=T\wedge\inf\{t\ge\tau:\int_\tau^t|\mathfrak b_s^u|^2ds\ge j\}$.
The stopped Brownian integral is square integrable; the compensated Poisson integral is uniformly integrable. Taking conditional expectations in \eqref{eq:verification-local-identity} gives
\begin{equation}\label{eq:verification-stopped-identity}
 \mathbb E_\tau Y_{\rho_j}=\xi^\top\overline P_\tau\xi+
                           \mathbb E_\tau\int_\tau^{\rho_j}q_sds.
\end{equation}
These stopping times equal $T$ eventually on each path, and
\[
 |Y_\theta|\le C\sup_{\tau\le s\le T}|X_s|^2+
                 C\int_\tau^T|u_s|^2ds=:Z,\qquad \mathbb EZ<\infty.
\]
Thus $Y_{\rho_j}\to Y_T$ in $L^1$ and $\int_\tau^{\rho_j}q_sds\uparrow\int_\tau^Tq_sds$.
Ordinary expectation in \eqref{eq:verification-stopped-identity} first shows that the last integral is integrable. Conditional expectations then converge in $L^1$, giving
\begin{equation}\label{eq:candidate-completion}
 J(\tau,\xi;u)=\xi^\top\overline P_\tau\xi+
 \mathbb E_\tau\int_\tau^T
 (u+\overline{\mathscr Q}X_-)^\top\overline{\mathscr N}
 (u+\overline{\mathscr Q}X_-)ds.
\end{equation}
Taking the essential infimum gives $\overline P_t\preceq P_t$ at deterministic times. Take a common null set over rational $t<T$. Right continuity extends the order to $[0,T)$, and $\overline P_T=P_T=G$ gives it at $T$. With $\overline P=P$, this is \eqref{eq:full-completion}. For a feedback-admissible candidate, substitute its admissible control $u=-\overline{\mathscr Q}X_-$ to obtain
\[
 V(\tau,\xi)\le J(\tau,\xi;u)=\xi^\top\overline P_\tau\xi
 \le V(\tau,\xi).
\]
Hence every such candidate represents the value kernel.

\textbf{Step 3. Uniqueness.}
Apply the preceding identity at rational $t<T$ to $\xi=e_i,e_i+e_j$. Polarization, right continuity, and $\overline P_T=P_T=G$ give $\overline P=P$ up to indistinguishability. Uniqueness of the canonical decomposition and Brownian--Poisson representation yields
\[
 \overline\Lambda=\Lambda\quad ds\otimes d\mathbb P\text{-a.e.},\qquad
 \overline\Xi=\Xi\quad ds\otimes\nu(de)\otimes d\mathbb P\text{-a.e.}
\]
Finally, for any optimal $u$, \eqref{eq:full-completion} implies
$\mathbb E_\tau\int_\tau^T|u+\mathscr QX_-|^2ds=0$.
Pathwise uniqueness of the closed-loop equation identifies its state and control with $(X^*,u^*)$.
\label{end:thm:verification}
\end{proof}

For a general bounded strongly regular candidate, the next estimate gives a direct sufficient condition for feedback admissibility.
\begin{prop}\label{prop:gain-bounds}
Under the bounded-coefficient assumptions, every bounded strongly regular SREJ solution satisfies, for some deterministic $C_*$,
\begin{equation}\label{eq:gain-bound}
 |\mathscr Q(s)|^2\le C_*\bigl(1+|\Lambda(s)D(s)|^2\bigr)
 \qquad ds\otimes d\mathbb P\text{-a.e.}
\end{equation}
In particular, if $D=0$ or $\Lambda D$ is essentially bounded under $ds\otimes d\mathbb P$, then the gain is bounded and its feedback is admissible for all square-integrable initial states.
\end{prop}
\begin{proof}
The jump bound \eqref{eq:Xi-bounded}, bounded coefficients, and finite intensity give $\mathscr M=m_0+\Lambda D$ with bounded $m_0$. Uniform positivity bounds $\mathscr N^{-1}$ and proves \eqref{eq:gain-bound}. If the gain is bounded, the closed-loop coefficients are bounded. Lemma~\ref{lem:state} then proves the $\mathcal S^2$ estimate, and boundedness of the gain gives square integrability of the feedback.
\end{proof}

\section{Examples and an Indefinite Portfolio Problem}\label{sec:application}

Examples~\ref{ex:rankloss}, \ref{ex:random-kernel} and \ref{ex:singular-reference} illustrate, respectively, rank loss of an optimal-state matrix, a bounded random gain with a nonzero Poisson martingale integrand, and a singular zero-control jump multiplier. Section~\ref{rev:finance} relates a self-financing portfolio model to the convexity condition. In these three examples, $N=(N_t)_{0\le t\le T}$ is a unit-rate Poisson process, $\widetilde N_t=N_t-t$, and the mark space is a singleton; the portfolio model uses the marked measure $\mu$.

\subsection{An optimal-state matrix can lose rank at a jump}
\begin{exam}\label{ex:rankloss}
Let $\kappa_0>0$ and consider
\begin{equation}\label{eq:rankloss-model}
 dX=u\,d\widetilde N,\qquad
 Q=S=1,\quad R=1+\kappa_0,\quad G=-\kappa_0.
\end{equation}
Thus $A=B=C=D=E=0$, $F=1$, and the zero-control state started from $1$ stays equal to $1$. For $\tau\in\mathcal T$ and $X(\tau)=x\in L^2(\mathcal F_\tau;\mathbb R)$, conditional It\^o isometry gives
\begin{equation}\label{eq:rankloss-cost}
 J(\tau,x;u)=-\kappa_0x^2+
       \mathbb E_\tau\int_\tau^T(X(s-)+u(s))^2ds.
\end{equation}
For a deterministic $t$ and zero initial state, put $v=X_-+u$ and
$U(s)=\mathbb E_t\int_t^s u^2dr$. Then $\mathbb E_tX(s)^2=U(s)$ and
\[
 U(s)\le2\mathbb E_t\int_t^T v^2dr+2\int_t^s U(r)dr.
\]
Applying Gronwall's inequality to the preceding conditional estimate gives
\begin{equation}\label{eq:rankloss-UC}
 J(t,0;u)\ge\tfrac12e^{-2T}\mathbb E_t\int_t^T u^2ds.
\end{equation}
Thus \UC{} holds.

Substitution into \eqref{eq:introN}, \eqref{eq:HMN} and \eqref{eq:SREJ} yields
\[
 P=-\kappa_0,\qquad\Lambda=\Xi=0,\qquad
 \mathscr N=1,\quad\mathscr M=1,\quad\mathscr H=1,\quad\mathscr Q=1.
\]
The unique optimizer is $u^*=-X^*_{-}$, and
\begin{equation}\label{eq:rankloss-state}
 X^*(s)=e^{s-\tau}x\,\mathbf1_{\{N_s-N_\tau=0\}},\qquad s\ge\tau.
\end{equation}
Its cost is $-\kappa_0x^2$ by \eqref{eq:rankloss-cost}. At the first jump after $\tau$ the optimal closed-loop jump multiplier is
\begin{equation}\label{eq:rankloss-multiplier}
 1+E-F\mathscr Q=1-1=0,
 \qquad\text{whereas}\qquad 1+E=1.
\end{equation}
For deterministic $t,h$ with $0<h\le T-t$ and initial state $X_t^*=1$,
\[
 \mathbb P(X_{t+h}^*=0)=\mathbb P(N_{t+h}-N_t\ge1)=1-e^{-h}>0.
\]
\label{end:ex:rankloss}
\end{exam}

\subsection{An indefinite portfolio problem}\label{rev:finance}
Consider a frictionless market consisting of a money-market account $S^0$ and a risky asset $S^1$:
\[
 \begin{aligned}
 dS_t^0&=r_tS_t^0\,dt, & S_0^0&=1,\\
 dS_t^1&=S_{t-}^1\left[r_t\,dt+\sigma_t\,dW_t+
                \int_{\mathbb Z}\gamma_t(e)\,\widetilde\mu(dt,de)\right],
 & S_0^1&=s_1>0.
 \end{aligned}
\]
The coefficients are bounded and predictable, with $0\le r_t\le\overline r$ and
$\gamma_t(e)>-1$ $dt\otimes\nu(de)\otimes d\mathbb P$-a.e.; hence $S^1$ remains strictly positive.
The two assets have the same return drift $r_t$, so the risky asset has zero excess return under $\mathbb P$.
Let $X$ denote total wealth and let $\pi_t$ be the monetary amount invested in the risky asset immediately before time $t$.
The amount in the money-market account is $X_{t-}-\pi_t$, and the corresponding holdings are
\[
 h_t^0=\frac{X_{t-}-\pi_t}{S_t^0},\qquad
 h_t^1=\frac{\pi_t}{S_{t-}^1}.
\]
For $\tau\in\mathcal T$, $\xi\in L^2(\mathcal F_\tau;\mathbb R)$ and
$\pi\in\mathcal U[\tau,T]$ with $m=1$, the self-financing relation
gives
\begin{equation}\label{eq:portfolio-state}
 \begin{aligned}
 dX_t&=h_t^0dS_t^0+h_t^1dS_t^1\\
     &=r_tX_{t-}\,dt+\sigma_t\pi_t\,dW_t+
          \int_{\mathbb Z}\gamma_t(e)\pi_t\,\widetilde\mu(dt,de),
 \qquad X_\tau=\xi.
 \end{aligned}
\end{equation}
Borrowing and short selling are allowed, and no lower bound on wealth is imposed.
Set
$q_t=\sigma_t^2+\int_{\mathbb Z}\gamma_t(e)^2\nu(de)\le\overline q$,
where $\overline r,\overline q$ are deterministic constants.
For $\lambda,\alpha>0$, minimize
\begin{equation}\label{eq:portfolio-cost}
 J(\tau,\xi;\pi)=\mathbb E_\tau\left[
       -\frac\lambda2X_T^2+\frac\alpha2\int_\tau^T\pi_t^2dt\right].
\end{equation}
The criterion rewards the terminal second moment with weight $\lambda$ and penalizes cumulative squared risky holdings with weight $\alpha$.
The correspondence with the LQ coefficients is
\[
 \begin{gathered}
 A=r,\qquad B=C=E=0,\qquad D=\sigma,\qquad F=\gamma,\\
 Q=S=0,\qquad R=\alpha/2,\qquad G=-\lambda/2<0.
 \end{gathered}
\]
The scalar SREJ has the explicit form
\begin{equation}\label{eq:portfolio-general-Riccati}
 \begin{gathered}
 \mathscr N_t=\frac\alpha2+P_{t-}\sigma_t^2+
                 \int_{\mathbb Z}(P_{t-}+\Xi_t(e))\gamma_t(e)^2\nu(de),\\
 \mathscr M_t=\Lambda_t\sigma_t+\int_{\mathbb Z}\Xi_t(e)\gamma_t(e)\nu(de),\\
 dP_t=-\left(2r_tP_{t-}-\frac{\mathscr M_t^2}{\mathscr N_t}\right)dt+
           \Lambda_t dW_t+\int_{\mathbb Z}\Xi_t(e)\widetilde\mu(dt,de),\\
 P_T=-\frac\lambda2,
 \qquad\pi_t^*=-\frac{\mathscr M_t}{\mathscr N_t}X_{t-}^*.
 \end{gathered}
\end{equation}

{\paragraph{Random interest rates: a sufficient condition}
For $\beta_t=\exp(\int_0^t r_sds)$, the zero-initial-state terminal wealth is
\begin{equation}\label{eq:portfolio-correct-factor}
 X_T^0=\beta_T\left[
  \int_\tau^T\beta_t^{-1}\sigma_t\pi_t\,dW_t+
  \int_\tau^T\int_{\mathbb Z}\beta_t^{-1}\gamma_t(e)\pi_t\,\widetilde\mu(dt,de)
 \right].
\end{equation}
It\^o's formula for $e^{-2\overline rt}|X_t^0|^2$ and conditional expectation give
\[
 \mathbb E_\tau|X_T^0|^2
 \le \overline q\,\mathbb E_\tau\int_\tau^T
          e^{2\overline r(T-t)}\pi_t^2dt
 \le e^{2\overline rT}\overline q\,
         \mathbb E_\tau\int_\tau^T\pi_t^2dt.
\]
Consequently,
\begin{equation}\label{eq:portfolio-random-UC}
 \alpha>\lambda e^{2\overline rT}\overline q
\end{equation}
is a sufficient condition for \UC{}, with
$\delta=(\alpha-\lambda e^{2\overline rT}\overline q)/2$.}

\paragraph{Deterministic interest rates: an exact calculation}
Now let $r$ be deterministic; $\sigma$ and $\gamma$ may remain random. Write
$a_t=\exp(\int_t^T r_sds)$ and $w_t=\alpha-\lambda a_t^2q_t$.
The factors $a_t$ are deterministic, so conditional martingale orthogonality and isometry give the exact identity
\begin{equation}\label{eq:portfolio-exact-cost}
 J(\tau,\xi;\pi)=-\frac\lambda2a_\tau^2\xi^2+
              \frac12\mathbb E_\tau\int_\tau^T w_t\pi_t^2dt.
\end{equation}
Moreover,
\[
 \mathbb E_\tau X_T=a_\tau\xi,\qquad
 \operatorname{Var}(X_T\mid\mathcal F_\tau)
 =\mathbb E_\tau\int_\tau^T a_t^2q_t\pi_t^2dt.
\]
Thus \UC{} holds precisely when
$\operatorname*{ess\,inf}_{dt\otimes d\mathbb P}w>0$.
In that case $\pi^*=0$: the optimal portfolio holds its entire wealth in the money-market account,
and the Riccati solution is
\begin{equation}\label{eq:portfolio-Riccati}
 P_t=-\frac\lambda2a_t^2,\qquad \Lambda=\Xi=0,\qquad
 \mathscr M=0,\quad\mathscr N=\frac12w_t,\quad\mathscr Q=0.
\end{equation}
Indeed, $dP_t=-2r_tP_tdt$ and
$X_t^*=\xi\exp(\int_\tau^t r_sds)$, $\tau\le t\le T$. 

\paragraph{The threshold and its equality case}
Let $r\ge0$ and $q>0$ be constant, with $q=\sigma^2+\int_{\mathbb Z}\gamma(e)^2\nu(de)$. At time zero the classification is
\begin{equation}\label{eq:portfolio-threshold}
 \begin{array}{ll}
 \alpha>\lambda q e^{2rT}:&\text{\UC{} holds; the unique optimum is }\pi=0,\\
 \alpha<\lambda q e^{2rT}:&\text{the cost is unbounded below},\\
 \alpha=\lambda q e^{2rT}:&\text{\UC{} fails, but the cost is bounded below.}
 \end{array}
\end{equation}
The second claim follows by scaling a control supported where $w_t<0$, a time set of positive measure. For a fixed later starting time $t<T$, the corresponding threshold is $\lambda q e^{2r(T-t)}$. At equality, $w_t\ge0$. If $r>0$, it is strictly positive for almost every $t>0$, so the zero control is still uniquely optimal although there is no uniform lower bound. If $r=0$, $w\equiv0$ and all admissible controls have the same cost.

\label{end:finance}

\subsection{A bounded random gain with a nonzero Poisson integrand}
\begin{exam}\label{ex:random-kernel}
Let $Z_t=(-1)^{N_t}$ and write $z=Z_{t-}\in\{-1,1\}$. Define
\begin{equation}\label{eq:random-example-data}
 \begin{gathered}
 dX_t=u_t\,d\widetilde N_t,\qquad
 k(z)=\frac{1+z}{2},\qquad n(z)=2-\frac z2,\qquad p(z)=-2+\frac z2,\\
 R=4,\qquad Q_t=z+n(z)k(z)^2,\qquad
 S_t=z+n(z)k(z),\qquad G=p(Z_T).
 \end{gathered}
\end{equation}
Here $A=B=C=D=E=0$ and $F=1$. The coefficients are bounded and predictable, the terminal weight is bounded and negative, and $Q=-1$ in the regime $z=-1$. 

A direct solution is
\begin{equation}\label{eq:random-example-solution}
 \begin{gathered}
 P_t=p(Z_t),\qquad\Lambda_t=0,\qquad\Xi_t=-Z_{t-},\\
 \mathscr N_t=n(Z_{t-}),\quad
 \mathscr M_t=n(Z_{t-})k(Z_{t-}),\quad\mathscr Q_t=k(Z_{t-}).
 \end{gathered}
\end{equation}
To verify it, note that $dP_t=-Z_{t-}dt-Z_{t-}d\widetilde N_t$, and
\[
 R+P_-+\Xi=2-\tfrac12z=n(z),\qquad
 S+\Xi=n(z)k(z),\qquad
 -Q+\frac{\mathscr M^2}{\mathscr N}=-z.
\]
These identities verify the Riccati equation.

Applying It\^o's formula to $P_tX_t^2$ and adding the running cost gives, for every admissible $u$,
\begin{equation}\label{eq:random-example-cost}
 J(\tau,\xi;u)=p(Z_\tau)\xi^2+
 \mathbb E_\tau\int_\tau^T n(Z_{s-})
        [u_s+k(Z_{s-})X_{s-}]^2ds.
\end{equation}
For zero initial state and deterministic starting time, set $v=u+k(Z_-)X_-$. We have $|k|\le1$ and $\mathbb E_tX(s)^2=\mathbb E_t\int_t^s u^2dr$. The Gronwall argument used for \eqref{eq:rankloss-UC} gives
\[
 \mathbb E_t\int_t^T u^2ds\le2e^{2T}\mathbb E_t\int_t^T v^2ds.
\]
Since $n\ge3/2$, \UC{} holds with $\delta=\frac34e^{-2T}$. The unique optimal control is $u^*=-k(Z_-)X^*_{-}$, and Proposition~\ref{prop:gain-bounds} gives admissibility because $D=0$. Moreover,
\begin{equation}\label{eq:random-example-exponential}
 \mathbb E\exp\left(C_0\int_0^T|\mathscr Q_t|^2dt\right)
 \le e^{C_0T}<\infty\qquad\text{for every }C_0>0.
\end{equation}

\label{end:ex:random-kernel}
\end{exam}

\subsection{A singular zero-control jump multiplier}
\begin{exam}\label{ex:singular-reference}
Let the Poisson rate be one, and consider
\begin{equation}\label{eq:singular-reference-model}
 dX=(-X_-+u)d\widetilde N,\qquad
 Q=\frac54,\quad S=-\frac12,\quad R=2,\quad G=-1.
\end{equation}
All the standing boundedness conditions hold, but $E=-1$, $F=1$, so $1+E=0$ identically. The zero-control state started from $1$ at $\tau$ is
\begin{equation}\label{eq:singular-reference-flow}
 \mathsf F_{\tau}(s)=e^{s-\tau}\mathbf1_{\{N_s-N_\tau=0\}}.
\end{equation}
Viewed as a scalar state-transition matrix, it vanishes after the first subsequent jump and is not globally invertible.

The constant Riccati solution and its coefficients are
\begin{equation}\label{eq:singular-reference-solution}
 \begin{gathered}
 P=-1,\quad\Lambda=\Xi=0,\qquad \mathscr N=1,\\
 \mathscr M=(-1)(-1)-\tfrac12=\tfrac12,\qquad
 \mathscr H=-1+\tfrac54=\tfrac14.
 \end{gathered}
\end{equation}
Thus $\mathscr H-\mathscr M^2/\mathscr N=0$ and the gain is $\mathscr Q=1/2$. Applying It\^o's formula to $-X^2$ and adding the running cost gives, for every admissible control,
\begin{equation}\label{eq:singular-reference-cost}
 J(\tau,\xi;u)=-\xi^2+
        \mathbb E_\tau\int_\tau^T(u_s+\tfrac12X_{s-})^2ds.
\end{equation}
For \UC{}, start from zero at deterministic $t$, put $v=u+X_-/2$ and $m(s)=\mathbb E_t X(s)^2$. Compensated-Poisson isometry gives
\[
 m(s)=\mathbb E_t\int_t^s(v_r-\tfrac32X_{r-})^2dr
 \le2\mathbb E_t\int_t^T v_r^2dr+\tfrac92\int_t^s m(r)dr.
\]
Gronwall gives $m(s)\le2e^{9T/2}\mathbb E_t\int_t^T v^2dr$. Therefore
\begin{equation}\label{eq:singular-reference-UC}
 \mathbb E_t\int_t^T u^2ds
 \le\bigl(2+Te^{9T/2}\bigr)\mathbb E_t\int_t^T v^2ds,
 \qquad \delta=\bigl(2+Te^{9T/2}\bigr)^{-1}>0.
\end{equation}
The unique optimizer is $u^*=-X_-^*/2$, and its state is explicitly
\begin{equation}\label{eq:singular-reference-optimal}
 X^*(s)=\xi e^{3(s-\tau)/2}(-\tfrac12)^{N_s-N_\tau}.
\end{equation}
Its supremum is at most $e^{3T/2}|\xi|$, giving admissibility for every $\xi\in L^2(\mathcal F_\tau;\mathbb R)$. 
\label{end:ex:singular-reference}
\end{exam}

\section{Conclusions}\label{sec:conclusion}

In this paper, we have studied indefinite stochastic LQ control with bounded random coefficients and finite-activity Poisson jumps. Under uniform convexity, we construct a bounded strongly regular solution of the associated Riccati equation and derive an admissible feedback representation of the unique open-loop optimal control. The constructed solution is maximal among bounded strongly regular solutions and is unique in the feedback-admissible class.

We obtain the control-weight lower bound by comparing the original and perturbed control problems, including their post-jump value kernels. Together with the matrix construction between jumps and the endpoint corrections, this allows singular jump multipliers in both the zero-control and optimal closed-loop systems. The examples illustrate these features and the boundary of uniform convexity. Extending the analysis to infinite-activity jumps is a natural further problem; it requires a different construction of the value process in place of the present interval-by-interval aggregation.

\Needspace{7\baselineskip}
\appendix
\section{It\^o's formula and localization}\label{app:product}

We record the expansion used in Lemma~\ref{lem:residual-measure} and Theorem~\ref{thm:verification}, together with the martingale integrability needed for conditional expectations.
{Set $x=X_-$, $p=P_-$, $z(e)=\Xi(e)$, $h(e)=E(e)x+F(e)u$, and $g=Cx+Du$. The Brownian integrand in It\^o's formula for $X^\top PX$ is
\begin{equation}\label{eq:product-brownian}
 \mathfrak b=x^\top\Lambda x+2x^\top p g,
\end{equation}
and the integrand of the compensated Poisson integral is
\begin{equation}\label{eq:product-jump}
 \mathfrak j(e)=(x+h(e))^\top(p+z(e))(x+h(e))-x^\top p x.
\end{equation}
After subtracting the first-order contributions already contained in $dP$ and $dX$, the remaining jump drift is
\begin{align}\label{eq:product-jump-correction}
 \mathfrak j(e)-2x^\top p h(e)-x^\top z(e)x
 &=x^\top z(e)h(e)+h(e)^\top z(e)x
       +h(e)^\top(p+z(e))h(e)\notag\\
 &=2x^\top z(e)h(e)+h(e)^\top(p+z(e))h(e).
\end{align}
The second equality uses symmetry. Substituting $h=Ex+Fu$ separates the state, cross and control terms:
\begin{align}\label{eq:product-jump-blocks}
 2x^\top zh+h^\top(p+z)h
 ={}&x^\top[zE+E^\top z+E^\top pE+E^\top zE]x\notag\\
 &+2x^\top[E^\top pF+(I_n+E)^\top zF]u\notag\\
 &+u^\top F^\top(p+z)F u.
\end{align} The Brownian terms contribute $C^\top pC+\Lambda C+C^\top\Lambda$ to the state block. Their cross and control blocks are $C^\top pD+\Lambda D$ and $D^\top pD$, respectively. These expansions give $\mathscr N,\mathscr M$ in \eqref{eq:introN} and $\mathscr H$ in \eqref{eq:HMN}.}

Consequently, the sum of the quadratic form and the accumulated running cost has the local decomposition
\begin{align}\label{eq:full-product-formula}
 d\left(X^\top PX+\int_\tau^{\cdot}\ell(s,X_s,u_s)ds\right)
 ={}&x^\top d\mathcal B\,x+
       (x^\top\mathscr Hx+2x^\top\mathscr M u+u^\top\mathscr Nu)ds\notag\\
 &+\mathfrak b\,dW+\int_{\mathbb Z}\mathfrak j(e)\widetilde\mu(ds,de).
\end{align}
At a predictable jump of $\mathcal B$, the state does not jump, so its entire contribution is already $x^\top\Delta\mathcal B\,x$. At a Poisson atom, $\Delta\mathcal B=0$. 

{The formula is first applied after localization by the state norm, the total variation of $\mathcal B$, and the quadratic variation of the Brownian integral. The stopped Brownian integral is then square integrable. For the Poisson term, bounded $p,z,E,F$ imply
\[
 |\mathfrak j(e)|\le C(|x|^2+|u|^2).
\]
Consequently,
\[
 \mathbb E\int_\tau^T\!\int_{\mathbb Z}|\mathfrak j(s,e)|\nu(de)ds
 \le C\mathbb E\!\left[\sup_{\tau\le s\le T}|X_s|^2+\int_\tau^T|u_s|^2ds\right]<\infty.
\]
The compensated Poisson integral is a uniformly integrable martingale. The same integrable bound controls the stopped cost in the verification argument.}

\vspace{6pt}
\noindent\textbf{Declarations.} The authors declare no conflict of interest. No empirical data were generated or analysed in this study.
\vspace{6pt}
\Needspace{8\baselineskip}
\label{loc:bibliography}

\end{document}